\documentclass[10pt]{amsart}
\usepackage{amsmath,amsthm,amscd,amssymb,amsfonts}
\usepackage{mathrsfs,upgreek}
\usepackage{hyperref}
\usepackage{mathtools}
\usepackage{graphicx}

\usepackage[sans]{dsfont} 
\usepackage[T1]{fontenc}
\DeclareMathAlphabet{\mathpzc}{OT1}{pzc}{m}{it} 

\numberwithin{equation}{section} 
\numberwithin{figure}{section} 

\theoremstyle{plain}

\newtheorem{prop}{Proposition}[section]

\newtheorem{lemm}[prop]{Lemma}
\newtheorem{sublemma}[prop]{Sublemma}
\newtheorem{theoalph}{Theorem}

\theoremstyle{definition}
\newtheorem{defi}[prop]{Definition}

\theoremstyle{remark}
\newtheorem{rema}[prop]{Remark}

\newtheoremstyle{citing}
  {3pt}
  {3pt}
  {\itshape}
  {}
  {\bfseries}
  {.}
  {.5em}
  {\thmnote{#3}}

\theoremstyle{citing}
\newtheorem*{generic}{}

\newcommand{\C}{\mathbb{C}}
\newcommand{\D}{\mathbb{D}}

\newcommand{\N}{\mathbb{N}}

\newcommand{\R}{\mathbb{R}}

\newcommand{\Z}{\mathbb{Z}}

\newcommand{\cE}{\mathcal{E}}

\newcommand{\cI}{\mathcal{I}}

\newcommand{\cK}{\mathcal{K}}

\newcommand{\cM}{\mathcal{M}}

\newcommand{\cO}{\mathcal{O}}
\newcommand{\cP}{\mathcal{P}}

\newcommand{\cR}{\mathcal{R}}

\newcommand{\cW}{\mathcal{W}}
\newcommand{\cX}{\mathcal{X}}

\newcommand{\fD}{\mathfrak{D}}

\newcommand{\sF}{\mathscr{F}}

\newcommand{\sM}{\mathscr{M}}

\newcommand{\sP}{\mathscr{P}}

\newcommand{\hJ}{\widehat{J}}

\newcommand{\hP}{\widehat{P}}

\newcommand{\hU}{\widehat{U}}
\newcommand{\hV}{\widehat{V}}
\newcommand{\hW}{\widehat{W}}
\newcommand{\hX}{\widehat{X}}

\newcommand{\hrho}{\widehat{\rho}}

\newcommand{\hSigma}{\widehat{\Sigma}}

\newcommand{\tI}{\widetilde{I}}
\newcommand{\tJ}{\widetilde{J}}

\newcommand{\tY}{\widetilde{Y}}

\newcommand{\talpha}{\widetilde{\alpha}}
\newcommand{\tbeta}{\widetilde{\beta}}

\newcommand{\teta}{\widetilde{\teta}}

\newcommand{\tmu}{\widetilde{\mu}}

\newcommand{\tPi}{\widetilde{\Pi}}
\newcommand{\trho}{\widetilde{\rho}}

\newcommand{\partn}[1]{{\smallskip \noindent \textbf{#1.}}}
\renewcommand{\=}{\coloneqq}
\newcommand{\dd}{\hspace{1pt}\operatorname{d}\hspace{-1pt}}

\DeclareMathOperator{\diam}{diam}

\DeclareMathOperator{\supp}{supp} 

\DeclareMathOperator{\crit}{crit}

\newcommand{\pV}{V_f}
\newcommand{\pP}{P_{f,n+1}(0)}
\newcommand{\pL}{L_f}
\newcommand{\pF}{F_f}
\newcommand{\pD}{D_f}

\newcommand{\pchicrit}{\chi_{\crit}(f)}

\newcommand{\psP}{\mathscr{P}_f}
\newcommand{\pfD}{\mathfrak{D}_f}

\newcommand{\df}{\widehat{f}_\lambda}

\newcommand{\signs}{\{ +, - \}^{\N}}
\newcommand{\uvarsigma}{\underline{\varsigma}}
\newcommand{\fs}{f_{\uvarsigma}}

\newcommand{\hcX}{\widehat{\cX}}

\newcommand{\wtp}{p^+}

\newcommand{\whf}{\widehat{f}}
\newcommand{\whm}{\widehat{m}}
\newcommand{\whp}{p^-}

\newcommand{\whx}{\widehat{x}}
\newcommand{\uwhx}{\underline{\whx}}

\newcommand{\chicritf}{\chi_{\crit}(f)}
\newcommand{\chicritpf}{\chi_{\crit}(\df)}

\newcommand{\chiinfR}{\chi_{\operatorname{inf}}^{\R}}
\newcommand{\chiinfC}{\chi_{\operatorname{inf}}}
\newcommand{\cl}[1]{{\rm cl}({#1})}

\begin{document}

\title[Sensitive dependence at positive temperature]{Sensitive dependence of geometric Gibbs states \\ at positive temperature}
\author{Daniel Coronel}
\address{Daniel Coronel, Departamento de Matem{\'a}ticas, Universidad Andr{\'e}s Bello, Avenida Rep\'ublica~220, Santiago, Chile}
\email{alvaro.coronel@unab.cl}
\author{Juan Rivera-Letelier}
\address{Department of Mathematics, University of Rochester. Hylan Building, Rochester, NY~14627, U.S.A.}
\email{riveraletelier@gmail.com}
\urladdr{\url{http://rivera-letelier.org/}}

\begin{abstract}
  We give the first example of a smooth family of real and complex maps having sensitive dependence of geometric Gibbs states at positive temperature.
  This family consists of quadratic-like maps that are non-uniformly hyperbolic in a strong sense.
  We show that for a dense set of maps in the family the geometric Gibbs states do not converge at positive temperature.
  These are the first examples of non-convergence at positive temperature in statistical mechanics or the thermodynamic formalism, and answers a question of van Enter and Ruszel.
  We also show that this phenomenon is robust: There is an open set of analytic 2\nobreakdash-parameter families of quadratic-like maps that exhibit sensitive dependence of geometric Gibbs states at positive temperature. 
\end{abstract}

\maketitle

%
%

\section{Introduction}
The main problem in statistical mechanics and the thermodynamic formalism is to describe the set of Gibbs states for a given interaction or potential. 
When the interaction or potential depends on some parameters, one is interested in knowing how the set of Gibbs states changes as a function of these parameters.
In several natural models in statistical mechanics, Gibbs states depend continuously (in the weak* topology) on the temperature, even if the number of Gibbs states changes with the temperature, see for instance \cite{AizDumSid15} and references therein.
In the thermodynamic formalism, it is known that if the dynamical system has sufficient expansion and the potential is sufficiently regular then in good situations for every value of the temperature there is a unique Gibbs state that varies continuously with the potential, see for example \cite{Bow75,BurCliFisTho1703,Cli18,Klo1711,LiRiv14b,LiRiv14a,Prz18,PrzRiv11,PrzRiv1405,VarVia10} and references therein.

In this article we study the temperature dependence of Gibbs states at positive temperature of smooth maps for the corresponding geometric potentials.
The geometric potential is completely determined by the map.
It arises naturally in several important problems, like in the construction of physical measures, as in the pioneering work of Sina{\u{\i}}~\cite{Sin72}, Ruelle~\cite{Rue76}, and Bowen~\cite{Bow75}.
The pressure of the geometric potential is connected, among other things, to several multifractal spectra, and large deviations rate functions, see for example~\cite[Lemma~2]{BinMakSmi03}, \cite{ChuRivTak1610,GelPrzRam10,GelPrzRam16}, \cite[Theorems~1.2 and~1.3]{KelNow92}, \cite[Appendix~B]{PrzRiv11}, \cite{OhWin17} and references therein.
 
The presence of a critical point may prevent the system to have geometric Gibbs states.
For example, in~\cite{CorRiv13} and~\cite{CorRiv15b} we constructed examples of Collet-Eckmann maps in the quadratic family, such that for every temperature above a critical value there a unique geometric Gibbs state, and for every temperature smaller than the critical value there is no geometric Gibbs state.
However, in those examples the geometric Gibbs states converge to a limit measure as the temperature drops to the critical value.

The main result of this article shows that there are analytic families of quadratic-like maps for which  an arbitrarily small perturbation of the parameter can have a large effect on the positive-temperature geometric Gibbs states.
Furthermore, this phenomenon is robust: There is an open set of analytic 2-parameter families of quadratic-like maps that exhibit ``sensitive dependence'' of geometric Gibbs states. 
In particular, this provides the first example in statistical mechanics and the thermodynamic formalism of non-convergence of Gibbs states at positive temperature, thus answering a question raised by van Enter and Ruszel in~\cite{vEnRus07}.
Analogous examples at zero temperature were given in~\cite{CorRiv1708}, see also~\cite{BisGarThi18,ChaHoc10,CorRiv15a,vEnRus07} for examples of non-convergence in a symbolic setting.

To state our results more precisely, we recall the concept of quadratic-like maps of Douady and Hubbard~\cite{DouHub85a}.
Given simply connected subsets~$U$ and~$V$ of~$\C$ satisfying~$\cl{U} \subset V$, a holomorphic map~$f : U \to V$ is a \emph{quadratic-like map} if it is proper of degree~$2$.
Such a map has a unique point at which the derivative~$Df$ vanishes; it is the \emph{critical point of~$f$}.
The \textit{filled-in Julia set} of a quadratic-like map~$f : U \to V$ is
$$ K(f)
\=
\{ z \in U \mid \text{ for every integer~$n \ge 1$, $f^n(z) \in U$} \}. $$
The \emph{Julia set $J(f)$} of~$f$ is the boundary of~$K(f)$, and it coincides with the closure of the repelling periodic points of~$f$.

Given a quadratic-like map~$f$, denote by~$\sM_f$ the space of all probability measures on~$J(f)$ that are invariant by~$f$.
For~$\mu$ in~$\sM_f$ denote by~$h_\mu(f)$ the measure-theoretic entropy of~$\mu$, and for each~$t$ in~$\R$ put
$$ P_f(t) \= \sup \left\{ h_\mu(f) - t \int \log |Df| \dd \mu \mid \mu \in \sM_f \right\}. $$
It is the \emph{pressure of~$f|_{J(f)}$ for the potential~$-t \log |Df|$}.
A measure~$\mu$ realizing the supremum above is an \emph{equilibrium state of~$f|_{J(f)}$ for the potential~$- t \log |Df|$}  or a \emph{geometric Gibbs state}.

A quadratic-like map~$f : U \to V$ is \emph{real} if~$U$, $V$, and~$f$ are all invariant under complex conjugation.
Note that the critical point of such a map is real.
A real quadratic-like map with critical point~$c$ is \emph{essentially topologically exact}\footnote{In~\S 3.1 we extend the definition of \emph{essentially topologically exact map} to certain quadratic-like maps that are not necessarily real.} if~$f^2(c)$ is defined and is different from~$f(c)$, if~$f$ maps the interval~$I(f)$ bounded by~$f(c)$ and~$f^2(c)$ to itself, and if~$f|_{I(f)}$ is topologically exact.
For such a map~$f$ we consider both the interval map~$f|_{I(f)}$ and the complex map~$f$ acting on its Julia set~$J(f)$.

Let~$f$ be a real quadratic-like map that is essentially topologically exact.
Denote by~$\sM_f^{\R}$ the space of all probability measures on~$I(f)$ that are invariant by~$f$.
For~$\mu$ in~$\sM_f^{\R}$ we denote by~$h_\mu(f)$ the measure-theoretic entropy of~$\mu$, and for each~$t$ in~$\R$ we put
$$ P^{\R}_f(t) \= \sup \left\{ h_\mu(f) - t \int \log |Df| \dd \mu \mid \mu \in \sM_f^{\R} \right\}. $$
It is the \emph{pressure of~$f|_{I(f)}$ for the potential~$-t \log |Df|$}.
A measure~$\mu$ realizing the supremum above is an \emph{equilibrium state of~$f|_{I(f)}$ for the potential~$- t \log |Df|$}  or a \emph{geometric Gibbs state}.

\begin{defi}[Sensitive dependence of positive-temperature Gibbs states]
\label{d:sensitive dependence}
Let~$\Lambda$ be a topological space and~$(f_\lambda)_{\lambda \in \Lambda}$ a continuous family of real or complex quadratic-like maps.
The family~$(f_\lambda)_{\lambda \in \Lambda}$ has \emph{sensitive dependence of positive-temperature geometric Gibbs states}, if there are a parameter~$\lambda_0$ and a positive number~$t_{*, 0}$ such that for every  increasing sequence of positive number $(\delta_\ell)_{\ell\in \N}$ converging to~$1$ as~$\ell \to + \infty$, there are a parameter~$\lambda$ in~$\Lambda$ arbitrarily close to~$\lambda_0$  and a positive number~$t_*$ arbitrarily close to~$t_{*, 0}$ such that the following property holds: For each~$\beta \in (0,t_*)$ there is a unique equilibrium state~$\rho_\beta^{\R}(\lambda)$ of~$f|_{I(f_\lambda)}$ (resp. $\rho_\beta(\lambda)$ 
of~$f|_{J(f_\lambda)}$) for the potential~$-\beta \log |Df_\lambda|$ and 
if we put for every $\ell \in \N$,
$\beta_\ell \= t_*\cdot\delta_\ell $, then the sequence of equilibrium states~$\left( \rho_{\beta_\ell}^{\R} (\lambda) \right)_{\ell \in \N}$ (resp. $\left( \rho_{\beta_\ell} (\lambda) \right)_{\ell \in \N}$) does not converge.
\end{defi}

We note that in the above situation the map~$f|_{I(f_\lambda)}$ (resp.~$f|_{J(f_\lambda)}$) must have a ``phase transition'' at~$t = t_*$ in the statistical mechanics sense, see~\cite[Theorem~A]{PrzRiv1405} for the real case, and~\cite[Main Theorem]{PrzRiv11} for the complex case.
Recall that for a real number~$t_*$, the map~$f|_{I(f)}$ (resp.~$f|_{J(f)}$) has a \emph{phase transition at~$t = t_*$}, if the geometric pressure function~$P_f^{\R}$ (resp.~$P_f$) is not real analytic at~$t = t_*$.

Our main results are stated as Theorems~\ref{t:pressure estimates} and~\ref{t:temperature dependence} in~\S\ref{ss:main results}.
The following is a simple consequence of these results that is easier to state.

\begin{generic}[Sensitive Dependence at Positive Temperature]
There is an open subset~$\Lambda_0$ of~$\C$ intersecting~$\R$,  a holomorphic family of quadratic-like maps~$(\df)_{\lambda \in \Lambda_0}$ and 
a compact subset $\Lambda$ of $\Lambda_0\cap \R$ such that the following properties hold.
For every real parameter~$\lambda$ in~$\Lambda$ the map~$\df$ is real, and the family of real (resp. complex) maps~$(\df)_{\lambda \in \Lambda}$ has sensitive dependence of positive-temperature geometric Gibbs states.
Moreover, this sensitive dependence occurs at every parameter in~$\Lambda$, in the sense that every parameter can be considered as $\lambda_0$ in  Definition~\ref{d:sensitive dependence}. 
\end{generic}

The family of quadratic-like maps~$(\df)_{\lambda \in \Lambda_0}$ that we use to prove the Sensitive Dependence at Positive Temperature is given explicitly in~\cite[\S3.3]{CorRiv1708}.
The maps in the subfamily~$(\df)_{\lambda \in \Lambda}$ are transitive and non-uniformly hyperbolic in a strong sense.
For example, they satisfy the Collet-Eckmann condition with uniform constants, and moreover they have uniform ``goodness constants'' in the sense of~\cite[Definition~2.2]{BalBenSch15}, 
\emph{cf}. Proposition~\ref{p:transporting Peierls}.
So, the lack of expansion is not responsible for the sensitive dependence of positive-temperature geometric Gibbs states.

We prove that as the temperature drops to~$1/t_*$ the geometric Gibbs states of~$\df|_{I(\df)}$ (resp. $\df|_{J(\df)}$) oscillate between~2 periodic measures.
So, at certain temperatures the geometric Gibbs state is close to one of these periodic measures, and at other temperatures it is close to the other periodic measure, see Theorem~\ref{t:temperature dependence} in~\S\ref{ss:main results}. 

The conclusions of the Sensitive Dependence at Positive Temperature hold for an open set of holomorphic 2\nobreakdash-parameter families of quadratic-like maps, see \cite[Remark~3.4]{CorRiv1708}.
Thus, for quadratic-like maps, the sensitive dependence of positive-temperature Gibbs states is a robust phenomenon for 2\nobreakdash-parameter families.

Note that the Sensitive Dependence at Positive Temperature does not say anything about the behavior of the geometric Gibbs states of~$\whf_{\lambda_0}|_{I(\whf_{\lambda_0})}$ or~$\whf_{\lambda_0}|_{J(\whf_{\lambda_0})}$ for temperatures above~$1/t_{*, 0}$.
As explained in~\cite[Remark~3.5]{CorRiv1708}, our results show that the parameter~$\lambda_0$ can be chosen so that the geometric Gibbs states of~$\whf_{\lambda_0}|_{I(\whf_{\lambda_0})}$ (resp.~$\whf_{\lambda_0}|_{J(\whf_{\lambda_0})}$) converge as the temperature drops to~$1/t_{*, 0}$, and that~$\lambda_0$ can be chosen so that they do not converge.

As a direct consequence of Theorem~\ref{t:pressure estimates} in~\S\ref{ss:main results} we have that for every~$\lambda$ in~$\Lambda$ the map~$\df$ has a phase transition  for both pressure functions~$P_{\df}^\R(t)$ and~$P_{\df}(t)$.
More precisely, we have the following.
For a quadratic-like map~$f$ with critical point~$c$, put
$$
\chicritf \= \liminf_{n\to +\infty} \frac{1}{n} \log |Df^n(f(c))|.
$$

\begin{generic}[High-Order Phase Transitions Theorem]
  Let~$\Lambda_0$, $\Lambda$, and~$(\df)_{\lambda \in \Lambda_0}$ be as in the statement of the Sensitive Dependence at Positive Temperature.
  Then there are positive constants $A^+, A^-, B^+$, and $B^-$ and a real analytic function~$t_* \colon \Lambda_0 \to (0, + \infty)$ such that for every~$\lambda$ in $\Lambda$, we have for every~$t \ge t_*(\lambda)$
$$
P_{\df}^\R(t) =  P_{\df}(t) = - t \frac{\chicritpf}{2},
$$
and for every~$t$ in~$(0, t_*(\lambda))$ close to~$t_*(\lambda)$, we have
\begin{multline*}
- t \frac{\chicritpf}{2} +  2^{-\left(\frac{A^-}{\sqrt{t_*(\lambda) - t}} + B^-\right)^3}
\le
P_{\df}^\R(t)
\le
P_{\df}(t)
\\ \le
- t \frac{\chicritpf}{2} + 2^{-\left(\frac{A^+}{\sqrt{t_*(\lambda) - t}} - B^+ \right)^3}.
\end{multline*}
In particular, both~$P_{\df}^{\R}$ and~$P_{\df}$ are of class~$C^2$ at~$t = t_*$, but neither of these functions is real analytic at~$t = t_*(\lambda)$.
\end{generic}

These phase transitions resemble a Kosterlitz-Thouless singularity, see for example~\cite[\S6]{Her07} for background.
In fact, near the critical parameter the geometric pressure function behaves as~$x \mapsto \exp(- |x|^{-3/2})$ near~$x = 0$.
The phase transitions in~\cite[Main Theorem]{CorRiv15b} are similar, but different: The geometric pressure function behaves as~$x \mapsto \exp(- x^{-2})$ near~$x = 0$.

The phase transitions in the theorem above are of ``freezing type'', in the sense that after the phase transition the pressure is affine and equal to its asymptote.
Two recent examples of freezing phase transitions in a symbolic setting are in~\cite{BruLep13,BruLep15}.
The authors prove that above a critical temperature there is a unique Gibbs state with full support and that below the critical temperature the unique Gibbs state is the unique shift-invariant measure of a substitution subshift, like the Thue-Morse subshift.
In contrast, in the theorem above for every~$\lambda$ in $\Lambda$ and~$t \ge t_*(\lambda)$ there are no geometric Gibbs state of~$\df|_{I(\df)}$ (resp. $\df|_{J(\df)}$) for the potential~$-t \log|D\df|$, see Theorem~\ref{t:pressure estimates} in~\S3.
Note in particular that, as temperature drops to~$1/t_*(\lambda)$, none of the accumulation measures of the geometric Gibbs states is a Gibbs state.
See \cite{BruTerTod1711,CorRiv13,DiaGelRam14,Vel1704} and references therein for other examples of phase transitions in the thermodynamic formalism.

\subsection{Organization}
\label{ss:organization}
After recalling some basic notions of the quadratic family in~\S\ref{s:preliminaries}, in~\S\ref{s:main results} we state Theorems~\ref{t:pressure estimates} and~\ref{t:temperature dependence}, and prove the Sensitive Dependence at Positive Temperature assuming these results.
To state Theorems~\ref{t:pressure estimates} and~\ref{t:temperature dependence}, we recall in~\S\ref{ss:uniform families} the notion of ``uniform families'' of quadratic-like maps introduced in~\cite{CorRiv1708}, some general facts about conformal measures in~\S\ref{ss:conformal measures}, and in~\S\ref{ss:inducing scheme} the inducing scheme that was first introduced in~\cite{CorRiv13}. 

The proof of Theorems~\ref{t:pressure estimates} and~\ref{t:temperature dependence} are given in~\S\ref{s:proofs}.
After describing the subfamily of maps in~\S\ref{ss:subfamily}, we give the proof of Theorem~\ref{t:pressure estimates} in~\S\ref{ss:proof of pressure estimates}.
It relies on several results and concepts from~\cite{CorRiv1708}, including the Geometric Peierls Condition, that we briefly summarize in Appendix~\ref{s:estimating pressure} for the reader's convenience.
These results are used to estimate the pressure and conformal measures in terms of the postcritical series.
The proof of Theorem~\ref{t:temperature dependence} is given in~\S\ref{ss:proof of temperature dependence}.
Roughly speaking, the most difficult part is to estimate the postcritical series by a certain numerical series.
The necessary estimates of the numerical series, in Appendix~\ref{s:reduction}, are given in an abstract setting that is independent of the rest of the paper.

\subsection{Notes and references}
\label{ss:notes}
The proofs of Theorems~\ref{t:pressure estimates} and~\ref{t:temperature dependence} rely strongly on estimates for the pressure functions and conformal measures for
quasi-quadratic maps satisfying the Geometric Peierls Condition proved in \cite{CorRiv1708}.
See Appendix~\ref{s:estimating pressure} for the definition of this concept.
Indeed, we use the same family of quasi-quadratic maps $(\df)_{\lambda\in \Lambda_0}$ as in \cite{CorRiv1708} but here we choose a different set of parameters.
That is, we choose a different subset~$\Lambda$ of~$\Lambda_0$.

\subsection{Acknowledgments}
We would like to thank the referees for the comments and pointers to the literature.

The first named author acknowledges partial support from FONDECYT grant 1161221, and would like to thank the University of Rochester for its hospitality.
The second named author acknowledges partial support from NSF grant DMS-1700291, and would like to thank Universidad Andr{\'e}s Bello for its hospitality.

\section{Preliminaries}
\label{s:preliminaries}

The main objective of this section is to recall the definition of Yoccoz puzzle pieces in~\S\ref{ss:puzzles} and to introduce a subset
of parameters of the quadratic family in~\S\ref{ss:Parameters} that possesses the combinatorics needed for our results.

We use~$\N$ to denote the set of integers that are greater than or equal to~$1$ and~$\N_0 \= \N \cup \{ 0 \}$.
For a Borel measure~$\rho$ on~$\C$, denote by~$\supp(\rho)$ its support.
Given an open subset~$G$ of~$\C$ and a map~$f : G \to \C$ that is a biholomorphism onto its image, the \emph{distortion of~$f$} on a subset~$C$ of~$G$ is
$$ \sup_{x, y \in C} \frac{|Df(x)|}{|Df(y)|}. $$

\subsection{Quadratic polynomials, Green's functions, and B{\"o}ttcher coordinates}
\label{ss:quadratic polynomials}
In this subsection and the next we recall some basic facts about the dynamics of
complex quadratic polynomials, see for instance~\cite{CarGam93} or~\cite{Mil06}
for references.

For~$c$ in~$\C$ we denote by~$f_c$ the complex quadratic polynomial
$$ f_c(z) = z^2 + c, $$
and by~$K_c$ the \emph{filled Julia set} of $f_c$; that is, the set of all
points~$z$ in~$\C$ whose forward orbit under~$f_c$ is bounded in~$\C$.
The set~$K_c$ is compact and its complement is the connected set consisting of
all points whose orbit converges to infinity in the Riemann sphere.
Furthermore, we have $f_c^{-1}(K_c) = K_c$ and~$f_c(K_c) = K_c$. 
The boundary~$J_c$ of~$K_c$ is the \emph{Julia set of~$f_c$}.

For a parameter~$c$ in~$\C$, the \emph{Green's function of~$K_c$} is the function
$G_c:\C \to [0,+\infty)$ that is identically~$0$ on~$K_c$, and that
for~$z$ outside~$K_c$ is given by the limit,
\begin{equation}
\label{def:Green function}
  G_c(z) = \lim_{n\rightarrow +\infty} \frac{1}{2^n} \log |f_c^n(z)| > 0.
\end{equation}
The function~$G_c$ is continuous, subharmonic, satisfies~$G_c \circ f_c = 2G_c$ on~$\C$, and it is harmonic and strictly positive outside~$K_c$.
On the other hand, the critical values of~$G_c$ are bounded from above by~$G_c(0)$, and
the open set
$$ U_c \= \{z\in \C \mid G_c(z) > G_c(0)\} $$
is homeomorphic to a punctured disk.
Notice that $G_c(c)=2G_c(0)$, thus~$U_c$ contains~$c$ if $0\notin K_c$. 

By B{\"o}ttcher's Theorem there is a unique conformal representation
\[
 \varphi_c: U_c
\rightarrow
\{z\in \C \mid |z| > \exp (G_c(0)) \},
\]
and this map conjugates~$f_c$ to $z \mapsto z^2$.
It is called \emph{the B{\"o}ttcher coordinate of~$f_c$} and satisfies $G_c =
\log |\varphi_c|$.

\subsection{External rays and equipotentials}
\label{ss:rays and equipotentials}
Let~$c$ be in~$\C$.
For~$v > 0$ the \emph{equipotential~$v$ of~$f_c$} is by definition~$G_c^{-1}(v)$.
A \emph{Green's line of~$G_c$} is a smooth curve on the complement of~$K_c$ in~$\C$ that is orthogonal to the equipotentials of~$G_c$ and that is maximal with this property. 
Given~$t$ in~$\R / \Z$, the \emph{external ray of angle~$t$ of~$f_c$}, denoted by~$R_c(t)$, is the Green's line of~$G_c$ containing
$$ \{ \varphi_c^{-1}(r \exp(2 \pi i t)) \mid \exp(G_c(0))< r < +\infty \}. $$
By the identity~$G_c \circ f_c= 2G_c$, for each~$v > 0$ and each~$t$ in~$\R / \Z$ the map~$f_c$ maps the equipotential~$v$ to the equipotential~$2v$ and maps~$R_c(t)$ to~$R_c(2t)$.
For~$t$ in~$\R / \Z$ the external ray~$R_c(t)$ \emph{lands at a point~$z$}, if~$G_c : R_c(t) \to (0, + \infty)$ is a bijection and if~$G_c|_{R_c(t)}^{-1}(v)$ converges to~$z$ as~$v$ converges to~$0$ in~$(0, + \infty)$.
By the continuity of~$G_c$, every landing point is in $J_c = \partial K_c$.

The \emph{Mandelbrot set~$\cM$} is the subset of~$\C$ of those
parameters~$c$ for which~$K_c$ is connected.
The function
\[
 \begin{array}{cccl}
  \Phi : &\C \setminus \cM & \to & \C \setminus \cl{\D}\\
          &     c         &  \mapsto           & \Phi(c) \= \varphi_c(c)
 \end{array}
\]
is a conformal representation, see~\cite[VIII, \emph{Th{\'e}or{\`e}me}~1]{DouHub84}.
For~$v > 0$ the \emph{equipotential~$v$ of~$\cM$} is by definition
$$ \cE(v) \= \Phi^{-1}(\{z\in \C \mid |z| = v \}). $$ 
On the other hand, for~$t$ in~$\R / \Z$ the set
$$ \cR(t) \= \Phi^{-1}(\{r \exp(2 \pi i t) \mid r > 1 \}) $$
is called the \emph{external ray of angle~$t$ of~$\cM$}.
We say that $\cR(t)$ \emph{lands at a point~$z$} in~$\C$, if~$\Phi^{-1} (r \exp(2\pi i t))$ converges to~$z$ as $r \searrow 1$.
When this happens~$z$ belongs to~$\partial \cM$.

\subsection{The wake~$1/2$}
\label{ss:wake 1/2}
In this subsection we recall a few facts that can be found for example
in~\cite{DouHub84} or~\cite{Mil00c}.

The external rays~$\cR(1/3)$ and~$\cR(2/3)$ of~$\cM$ land at the parameter~$c = -3/4$, and these are the only external rays of~$\cM$ that land at this point, see
for example~\cite[Theorem~1.2]{Mil00c}.
In particular, the complement in~$\C$ of the set
$$ \cR(1/3) \cup \cR(2/3) \cup \{ - 3/4 \} $$
has~2 connected components; we denote by~$\cW$ the connected component
containing the point~$c = -2$ of~$\cM$.

For each parameter~$c$ in~$\cW$ the map~$f_c$ has~2 distinct fixed points; one
of the them is the landing point of the external ray~$R_c(0)$ and it is denoted
by~$\beta(c)$; the other one is denoted by~$\alpha(c)$.
The only external ray landing at~$\beta(c)$ is~$R_c(0)$, and
the only external ray landing at~$-\beta(c)$
is~$R_c(1/2)$.

Moreover, for every parameter~$c$ in~$\cW$ the only external rays 
of~$f_c$ landing at~$\alpha(c)$ are~$R_c(1/3)$
and~$R_c(2/3)$, see for example~\cite[Theorem~1.2]{Mil00c}.
The complement of~$R_c(1/3) \cup R_c(2/3) \cup \{ \alpha(c)
\}$ in~$\C$ has~2 connected components; one containing~$- \beta(c)$ and~$z =
c$, and the other one containing~$\beta(c)$ and~$z = 0$.
On the other hand, the point~$\alpha(c)$ has~2 preimages by~$f_c$: Itself and~$\talpha(c) \= - \alpha(c)$.
The only external rays landing at~$\talpha(c)$ are~$R_c(1/6)$
and~$R_c(5/6)$.

\subsection{Yoccoz puzzles and para-puzzle}
\label{ss:puzzles}
In this subsection we recall the definitions of Yoccoz puzzles and para-puzzle.
We follow~\cite{Roe00}.

\begin{defi}[Yoccoz puzzles]
Fix~$c$ in~$\cW$ and consider the open region $X_c^\circ \= \{z\in \C \mid G_c(z) <
1\}$. 
The \emph{Yoccoz puzzle of~$f_c$} is given by the following sequence of
graphs~$(I_{c, n})_{n = 0}^{+ \infty}$ defined for~$n = 0$ by:
\[
 I_{c,0} \= \partial X_c^\circ \cup (X_c^\circ \cap \cl{R_c(1/3)} \cap \cl{R_c(2/3)}),
\]
and for~$n \ge 1$ by~$I_{c,n} \= f_c^{-n}(I_{c,0})$.
The \emph{puzzle pieces of depth~$n$} are the connected components of $f_c^{-n}(X_c^\circ) \setminus I_{c,n}$.
The puzzle piece of depth~$n$ containing a point~$z$ is denoted by~$P_{c,n}(z)$.
\end{defi}

Note that for a real parameter~$c$, every puzzle piece intersecting the real line is invariant under complex conjugation.
Since puzzle pieces are simply-connected, it follows that the intersection of such a puzzle piece with~$\R$ is an interval.

\begin{defi}[Yoccoz para-puzzle\footnote{In contrast to~\cite{Roe00}, we only consider the para-puzzle in the wake~$\cW$.}]
Given an integer~$n \ge 0$, put
$$ J_n
\=
\{t\in [1/3,2/3] \mid 2^n t ~ (\mathrm{mod}\, 1) \in \{1/3,2/3\} \}, $$
let~$\cX_n$ be the intersection of~$\cW$ with the open region in the parameter plane bounded by the equipotential~$\cE(2^{-n})$ of~$\cM$, and put
\[
 \cI_{n}
\=
\partial \cX_n \cup \left( \cX_n \cap \bigcup_{t\in J_n} \cl{\cR(t)} \right).
\]
Then the \emph{Yoccoz para-puzzle of~$\cW$} is the sequence of graphs~$(\cI_n)_{n = 0}^{+ \infty}$.
The \emph{para-puzzle pieces of depth~$n$} are the connected components of $\cX_n \setminus \cI_n$.
The para-puzzle piece of depth~$n$ containing a parameter~$c$ is denoted by~$\cP_n(c)$.  
\end{defi}
Observe that there is only~1 para-puzzle piece of depth~$0$, and only~1 para-puzzle piece of depth~$1$; they are bounded by the same external rays but different equipotentials.
Both of them contain~$c = - 2$.

Fix a parameter~$c$ in~$\cP_0(-2)$.
There are precisely~2 puzzle pieces of depth~$0$: $P_{c, 0}(\beta(c))$ and~$P_{c, 0}(-\beta(c))$.
Each of them is bounded by the equipotential~$1$ and by the closures of the
external rays landing at~$\alpha(c)$.
Furthermore, the critical value~$c$ of~$f_c$ is contained 
in~$P_{c, 0}(- \beta(c))$ and the critical point in~$P_{c, 0}(\beta(c))$.
It follows that the set~$f_c^{-1}(P_{c, 0}(\beta(c)))$ is the 
disjoint union of~$P_{c, 1}(- \beta(c))$ and~$P_{c, 1}(\beta(c))$, so~$f_c$ maps each of the sets~$P_{c, 1}(-\beta(c))$ and~$P_{c,
1}(\beta(c))$ biholomorphically to~$P_{c, 0}(\beta(c))$.
Moreover, there are precisely~3 puzzle pieces of depth~$1$: 
$$ P_{c, 1}(-\beta(c)),
P_{c, 1}(0)
\quad \text{and} \quad
P_{c, 1}(\beta(c)); $$
$P_{c, 1}(- \beta(c))$ is bounded by the equipotential~$1/2$ and by the 
closures of the external rays that land at~$\alpha(c)$; $P_{c, 1}(\beta(c))$ is
bounded by the equipotential~$1/2$ and by the closures of the external rays that
land at~$\talpha(c)$; and~$P_{c, 1}(0)$ is bounded by the equipotential~$1/2$
and by the closures of the external rays that land at~$\alpha(c)$ and
at~$\talpha(c)$.
In particular, the closure of~$P_{c, 1}(\beta(c))$ is 
contained in~$P_{c, 0}(\beta(c))$.
It follows from this that for each integer~$n \ge 1$ the map~$f_c^n$ maps~$P_{c, n}(- \beta(c))$ biholomorphically to~$P_{c, 0}(\beta(c))$.

The following lemma proved in~\cite[Lemma~3.3]{CorRiv13} is useful for the description of the parameter introduced in~\S\S\ref{ss:Parameters}.
\begin{lemm}
\label{l:auxiliary para-puzzle pieces}
For each integer~$n \ge 1$, the following properties hold.
\begin{enumerate}
\item[1.]
The para-puzzle piece~$\cP_n(-2)$ contains the closure of~$\cP_{n + 1}(-2)$.
\item[2.]
For each  parameter~$c$ in~$\cP_n(-2)$  the 
critical value~$c$ of~$f_c$ is in~$P_{c, n}(-\beta(c))$.
\end{enumerate}
\end{lemm}

\subsection{The uniformly expanding Cantor set}
\label{ss:expanding Cantor set}
For a parameter~$c$ in~$\cP_3(-2)$, the maximal invariant set~$\Lambda_c$ of~$f_c^3$ in~$P_{c, 1}(0)$ plays an important role in  the definition of the combinatorics of the quadratic-like maps used in the  proof of Theorem~\ref{t:temperature dependence}.
We recall here its definition using certain puzzle pieces that we also use in the next section.

Fix~$c$ in~$\cP_3(-2)$.
There are precisely~2 connected components of~$f_c^{-3}(P_{c, 1}(0))$ contained in~$P_{c, 1}(0)$ that we denote by~$Y_c$ and~$\tY_c$.
The closures of these sets are disjoint and contained in~$P_{c, 1}(0)$.
The sets~$Y_c$ and~$\tY_c$ are distinguished by the fact that~$Y_c$ contains in its boundary the common landing point of the external rays~$R_c(7/24)$ and~$R_c(17/24)$, denoted~$\gamma(c)$, and that~$\tY_c$ contains in its boundary the common landing point of the external rays~$R_c(5/24)$ and~$R_c(19/24)$.
The map~$f_c^3$ maps each of the sets~$Y_c$ and~$\tY_c$ biholomorphically to~$P_{c, 1}(0)$.
Thus, if we put
\[
\begin{array}{cccl}
  g_c & : Y_c \cup \tY_c & \to & P_{c,1}(0)\\
      &    z & \mapsto & g_c(z) \= f_c^{3}(z),
 \end{array}
\]
then
$$ \Lambda_c = \bigcap_{n\in \N} g_c^{-n}(\cl{P_{c,1}(0)}). $$

\subsection{Parameters}
\label{ss:Parameters}
The combinatorics of the quadratic-like maps that  we use to prove Theorem~\ref{t:temperature dependence} are modeled  from a subset  of the quadratic family introduced in~\cite[Proposition~3.1]{CorRiv13}.
In this subsection we recall the definition of this parameter set.

Given an integer~$n \ge 3$, let~$\cK_n$ be the set of all those real
parameters~$c<0$ such that 
$$ f_c(c) > f_c^2(c) > \cdots > f_c^{n-1}(c) > 0
\quad \text{and} \quad
f_c^n(c) \in \Lambda_c. $$
Note that for a parameter~$c$ in~$\cK_n$, the critical point of~$f_c$ cannot be
asymptotic to a non-repelling periodic point.
This implies that all the periodic points of~$f_c$ in~$\C$ are hyperbolic
repelling and therefore that~$K_c = J_c$, see~\cite{Mil06}.
On the other hand, we have~$f_c(c) > c$ and the interval~$I_c = [c, f_c(c)]$ is
invariant by~$f_c$.
This implies that~$I_c$ is contained in~$J_c$ and hence that for every real
number~$t$ we have~$P_c^{\R}(t) \le P_c(t)$.
Note also that~$f_c|_{I_c}$ is not renormalizable, so~$f_c$ is topologically
exact on~$I_c$, see for example~\cite[Theorem~III.4.1]{dMevSt93}.

Since for~$c$ in~$\cK_n$ the critical point of~$f_c$ is not periodic, for every integer~$k \ge 0$ we have~$f_c^{n + 3k}(c) \neq 0$.
Thus, we can define the sequence~$\iota(c)$ in~$\{0, 1 \}^{\N_0}$ for each~$k \ge 0$ by
$$ \iota(c)_k
\=
\begin{cases}
0 & \text{ if }  f_c^{n + 3k}(c) \in Y_c; \\
1  & \text{ if } f_c^{n + 3k}(c) \in \widetilde{Y}_c.
\end{cases} $$

The following proposition without the statement about the homeomorphism was proved in~\cite[Proposition~3.1]{CorRiv13}.
For the proof of the homeomorphism property see~\cite[Proposition~2.3]{CorRiv1708}.
\begin{prop}
\label{p:ps}
For each integer~$n \ge 3$, the set~$\cK_n$ is a compact subset of
$$ \cP_n(-2) \cap (-2, -3/4), $$
and the function~$\iota \colon \cK_n \to \{0, 1\}^{\N_0}$ is homeomorphism.
Finally, for each~$\delta > 0$ there is~$n_0 \ge 3$ such that for each
integer~$n \ge n_0$ the set~$\cK_n$ is contained in the interval~$(-2, -2 +
\delta)$.
\end{prop}

\section{Main results}
\label{s:main results}
In this section we state Theorem~\ref{t:temperature dependence}, and prove the Sensitive Dependence at Positive Temperature assuming this result.
We also state a more precise version of the High-Order Phase Transitions Theorem as Theorem~\ref{t:pressure estimates}.

In~\S\ref{ss:uniform families} we recall some generalities about uniform families of quadratic-like maps.  In~\S\ref{ss:conformal measures} we recall a general result about conformal measures and in~\S\ref{ss:inducing scheme} the inducing scheme that we use to estimate the pressure function and the conformal measures.
Then, we state Theorems~\ref{t:pressure estimates} and~\ref{t:temperature dependence} in~\S\ref{ss:main results}, and we assume this last result in~\S\ref{ss:proof of positive sensitive dependence} to prove the Sensitive Dependence at Positive Temperature.

\subsection{Uniform families of quadratic-like maps}
\label{ss:uniform families}
A quadratic-like map~$f \colon U \to V$ is \emph{normalized}, if its unique critical point is~$0$, and if~$D^2f(0) = 2$.
For such a map~$f$ there is a holomorphic function~$R_f \colon U \to \C$ such that for~$w$ in~$U$ we have
$$ f(w) = f(0) + w^2 + w^3 R_f(w). $$
Note that~$f$ is uniquely determined by its critical value~$f(0)$, and the function~$R_f$.

By the straightening theorem of Douady and Hubbard~\cite{DouHub85a}, for every quadratic-like map~$f \colon U \to V$ there is~$c$ in~$\C$ and a quasi-conformal homeomorphism~$h : \C \to \C$ that conjugates the quadratic polynomial~$f_c$ to~$f$ on a neighborhood of~$J_c$.
In the case~$f$ is real, $c$ is real, and~$h$ can be chosen so that it commutes with the complex conjugation.
In all the cases, the quasi-conformal homeomorphism~$h$ can be chosen to be holomorphic on a neighborhood of infinity, and tangent to the identity there.

Put
$$ \cX \= \{ c \in \C \mid G_c(c) \le 1 \}
\quad \text{and} \quad
\hcX \= \{ c \in \C \mid G_c(c) \le 2 \}, $$
and for~$c$ in~$\C$, put
$$ X_c \= \{ z \in \C \mid G_c(z) \le 1 \}
\quad \text{and} \quad
\hX_c \= \{ z \in \C \mid G_c(z) \le 2 \}. $$
Note that~$X_c$ is contained in the interior of~$\hX_c$, and that
$$ \cX = \{ c \in \C \mid c \in X_c \}
\quad \text{and} \quad
\hcX = \{ c \in \C \mid c \in \hX_c \}. $$
\begin{defi}[Uniform family of quadratic-like maps]
A family~$\sF$ of normalized quadratic-like maps is \emph{uniform}, if there are constants~$K\ge 1$ and~$R>0$, such that for each~$f$ in~$\sF$ there are~$c(f)$ in~$\cX$ and a $K$\nobreakdash-quasi-conformal homeomorphism~$h_f$ of~$\C$ satisfying the following properties.
\begin{enumerate}
\item[1.]
The homeomorphism~$h_f$ conjugates~$f_{c(f)}$ on~$\hX_{c(f)}$ to~$f$ on~$h_f(\hX_{c(f)})$.
Furthermore, if~$f$ is real, then~$h_f$ commutes with the complex conjugation.
\item[2.]
The set~$\hX_{c(f)}$ is contained in~$B(0, R)$, and the homeomorphism~$h_f$ is holomorphic on~$\C \setminus \cl{B(0, R)}$, and it is tangent to the identity at infinity.
\end{enumerate}
\end{defi}

Note that property~1 implies that~$h_f(0) = 0$.

Let~$\sF$ be a uniform family of quadratic-like maps.
For each~$f$ in~$\sF$ put
$$ X_f \= h_f(X_{c(f)})
\quad \text{and} \quad
\hX_f \= h_f(\hX_{c(f)}). $$
By the definition of uniform family, the puzzle pieces of~$f_{c(f)}$ can be push-forward to~$X_f$ by~$h_f$.
We call to these sets the \emph{puzzle pieces of~$f$}.
We say that a puzzle piece of $f$ has \textit{depth} $n$  if it is the push-forward of a puzzle piece of $c(f)$ with depth~$n$.
The puzzle piece of depth~$n$ of~$f$ containing~$w$ is denoted~$P_{f,n}(w)$.
Thus, we have
$$
P_{f,n}(w) \= h_f(P_{c(f),n}(h_f^{-1}(w))).
$$
 Set 
$$\beta(f) \= h_f(\beta(c(f)))
\quad \text{and} \quad
\tbeta(f) \= h_f(-\beta(c(f))).$$

For every  integer $n\ge 0$, put
$$\cP_n(\sF) \= \{f \in \sF \mid c(f) \in \cP_n(-2)\},$$  
and for
$n \ge 3$, put 
$$\cK_n(\sF) \= \{f \in \sF \mid c(f) \in \cK_n\}.$$  

Moreover, for~$f$ in~$\cP_3(\sF)$ put
\begin{displaymath}
  Y_f \= h_f(Y_{c(f)}),
  \text{ and }
  \tY_f \= h_f(\tY_{c(f)}),
\end{displaymath}
and let~$g_f \colon h_f(Y_{c(f)}\cup \tY_{c(f)}) \to P_{f, 1} (0)$ be defined by~$g_f \= h_f^{-1} \circ g_{c(f)}\circ h_f$.
Moreover, let~$p(f)$ and~$\wtp(f)$ be the unique fixed point of~$g_f$ in~$Y_f$ and $\tY_f$, respectively, and denote by~$\whp(f)$ the unique fixed point of~$g_f^2$ in~$\tY_f$ that is different from~$\wtp(f)$; it is a periodic point of~$g_f$ of minimal period~$2$.
Furthermore, denote by
\begin{displaymath}
  \cO^+(f) \= \left\{ f^j(p^+(f)) \mid j \in \{0, 1, 2 \} \right\}
  \text{ and }
  \cO^-(f) \= \left\{ f^j(p^-(f)) \mid j \in \{0, 1, \dots, 5 \} \right\}
\end{displaymath}
the orbits of~$p^+(f)$ and~$p^-(f)$ under~$f$, respectively.

For each integer~$n \ge 5$, and each~$f$ in~$\cK_n(\sF)$, put $\iota(f)\=\iota(c(f))$, see~\S\ref{ss:Parameters}, and note that for every integer $j\ge 0$ we have  
$$
\iota(f)_j \= \begin{cases}
0 & \text{ if } f^{n+1+3j}(0) \in Y_f;\\
1 & \text{ if } f^{n+1+3j}(0) \in \tY_f.
\end{cases} $$

Finally, for every~$f$ in~$\sF$ such that~$c(f)$ is real and belongs to~$[-2,0)$, denote by~$I(f)$ the image under~$h_f$ of the interval $[c(f), f_{c(f)}(c(f))]$.
Observe that $f(I(f))=I(f)$.
A quadratic-like map~$f$ in~$\sF$ is \emph{essentially topologically exact} if~$c(f)$ is in $[-2,0)$, and if~$f|_{I(f)}$ is topologically exact.
For such a map~$f$ we consider both, the map~$f|_{I(f)}$, and the complex map~$f$ acting on its Julia set~$J(f)$.
We also define~$\sM_f^{\R}$, $P_f^{\R}$, and \emph{equilibrium states} or \emph{Geometric Gibbs states} of~$f|_{I(f)}$ as in the introduction. 
In the case~$f$ is real, the definitions above coincide with those in the introduction.

\subsection{Conformal measures}
\label{ss:conformal measures}
Throughout this section we fix a uniform family of quadratic-like maps~$\sF$.

Let~$n \ge 4$ be an integer and~$f$  in~$\cK_n(\sF)$.
Given~$t > 0$ and a real number~$p$, a measure~$\mu$ is \emph{$(t, p)$\nobreakdash-conformal for~$f|_{I(f)}$} (resp.~$f$), if for every subset~$U$ of~$I(f)$ (resp.~$J(f)$) on which~$f|_{I(f)}$ (resp.~$f$) is injective we have
\begin{equation*}
\mu(f|_{I(f)}(U))
=
\exp(p) \int_U |Df|^{t} \dd \mu
\left( \text{resp. } \mu(f(U)) = \exp(p) \int_U |Df|^{t} \dd \mu \right).
\end{equation*}
In the case where~$P_f^{\R}(t) = 0$ (resp. $P_f^{\C}(t) = 0$), a $(t, 0)$\nobreakdash-conformal measure is simply called \emph{conformal}.

On the other hand, the \emph{conical} or \emph{radial Julia set of~$f|_{I(f)}$} (resp.~$f$) is the set of all points~$x$ in~$I(f)$ (resp.~$J(f)$) for which the following property holds: There exists~$r > 0$ and an unbounded sequence of positive integers~$(n_j)_{j = 1}^{+ \infty}$, such that for every~$j$ the map~$f|_{I(f)}^{n_j}$ (resp.~$f^{n_j}$) maps a neighborhood of~$x$ in~$I(f)$ (resp.~$J(f)$) diffeomorphically to~$B(f^{n_j}(x), r)$.

In the case where~$f$ is a real quadratic map, the following is~\cite[Proposition~5.1]{CorRiv15b}.
The proof applies without change to quadratic-like maps in~$\cK_n(\sF)$.

\begin{prop}[\cite{CorRiv15b}, Proposition~5.1]
\label{p:conformal measures}
Let~$n \ge 4$ be an integer, $f$ a map in~$\cK_n(\sF)$, and let~$t > 0$ and~$p$ in~$\R$ be given.
Then there is at most one $(t, p)$\nobreakdash-conformal probability measure of~$f|_{I(f)}$ (resp.~$f$) supported on~$I(f)$ (resp.~$J(f)$).
If such a measure~$\mu$ exists, then~$p \ge P_f^{\R}(t)$ 
(resp. $p \ge P_f^{\C}(t)$),
and~$\mu$ is either supported on the backward orbit of~$z = 0$ and dissipative, or~$\mu$ is nonatomic
and supported on the conical Julia set of~$f|_{I(f)}$ (resp.~$f$).
Furthermore, the former case holds precisely when the following series converges:
\begin{multline}
  \label{e:Poincare series}
 \sum_{j = 1}^{+ \infty} \exp(- j p) \sum_{y \in f|_{I(f)}^{-j}(0)} |Df^j (y)|^{-t}  
\\ 
\left( \text{resp. } \sum_{j = 1}^{+ \infty} \exp(- j p) \sum_{y \in f^{-j}(0)} |Df^j (y)|^{-t} \right).
\end{multline}
\end{prop}

\subsection{Inducing scheme and conformal measures}
\label{ss:inducing scheme}
Throughout this subsection we fix a uniform family of quadratic-like maps~$\sF$.
In this subsection we introduce the inducing scheme to estimate the geometric pressure function for maps in~$\cK_n(\sF)$.  
 
Let~$n \ge 5$ be an integer and~$f$  in~$\cK_n(\sF)$.
Put 
$$ 
\pV \= \pP = f^{-1}(P_{f, n}( \tbeta(f))) 
$$ 
and 
$$
\pD\=\{z\in \pV  \mid  f^m(z) \in V_f  \text{ for some } m\ge 1\}.
$$
 For~$w$ in~$\pD$ put~$m_f(w)\=\min\{m\in \N \mid f^m(w) \in \pV\}$, and call it the \emph{first return time of~$w$ to~$\pV$}.
The \emph{first return map to~$\pV$} is defined by
$$ \begin{array}{rcl}
F_f : \pD & \to & \pV \\
w & \mapsto & F_f(w) \= f^{m_f(w)}(w).
\end{array} $$
It is easy to see that~$\pD$ is a disjoint union of puzzle pieces; 
so each connected component of~$\pD$ is a puzzle piece.
Note furthermore that in each of these puzzle pieces~$W$, 
the return time function~$m_f$ is constant; denote the common value
of~$m_f$ on~$W$ by~$m_f(W)$.

Throughout the rest of this subsection we put $\hV_f \= P_{f,4}(0)$. 

\begin{lemm}[Uniform distortion bound, \cite{CorRiv1708}, Lemma~5.1]
\label{l:bounded distortion to nice}
There is a constant~$\Xi_1 > 1$ such that for each integer~$n \ge 5$ and
each
~$f$ in~$\cK_n(\sF)$ the following property holds: For every connected
component~$W$ of~$\pD$ the map~$F_f|_W$ is univalent and its distortion is
bounded by~$\Xi_1$.
Furthermore, the inverse of~$F_f|_W$ admits a univalent extension to~$\hV_f$
taking images in~$\pV$.
In particular, $F_f$ is uniformly expanding with respect to the hyperbolic
metric on~$\hV_f$.
\end{lemm}

Denote by~$\pfD$ the collection of connected components of~$\pD$ and
if $c(f)$ is real denote by~$\pfD^{\R}$ the sub-collection of~$\pfD$ of those sets intersecting~$I(f)$.
For each~$W$ in~$\pfD$ denote by~$\phi_W : \hV_f \to \pV$ 
the extension of~$f|_{W}^{-1}$ given by Lemma~\ref{l:bounded distortion to nice}. Given an
integer~$\ell \ge 1$ we denote by~$E_{f,\ell}$ (resp. $E_{f,\ell}^{\R}$) the
set of all words of length~$\ell$ in the alphabet~$\pfD$ (resp. $\pfD^{\R}$).
Again by~Lemma~\ref{l:bounded distortion to nice}, for each integer~$\ell \ge 1$ and
each word~$W_1 \cdots W_\ell$  in~$E_{f,\ell}$ the composition
$$ \phi_{W_1 \cdots W_\ell} = \phi_{W_1} \circ \cdots \circ \phi_{W_\ell} $$
is defined on~$\hV_f$.
We also put
$$ m_f(W_1 \cdots W_\ell) = m_f(W_1) + \cdots + m_f(W_\ell). $$

For~$t, p$ in~$\R$ and an integer~$\ell \ge 1$ put
$$ Z_{\ell}(t, p)
\=
\sum_{\underline{W} \in E_{f,\ell}} \exp(-m_f(\underline{W}) p) \left( \sup \{
|D\phi_{\underline{W}}(z) | \mid z \in \pV \} \right)^t $$
and
$$ Z_{\ell}^{\R}(t, p)
\=
\sum_{\underline{W} \in E_{f,\ell}^{\R}} \exp(-m_f(\underline{W}) p) \left(
\sup \{ |D\phi_{\underline{W}}(z) | \mid z \in \pV \} \right)^t. $$
For a fixed~$t$ and~$p$ in~$\R$ 
the sequence
$$ \left(\frac{1}{\ell} \log Z_{\ell}(t, p) \right)_{\ell = 1}^{+ \infty}
\left( \text{resp. } \left(\frac{1}{\ell}  \log Z_{\ell}^{\R}(t, p) \right)_{\ell = 1}^{+ \infty} \right) $$
converges to the pressure function of~$F_f$ (resp. $F_{f}|_{\pD \cap I(f)}$) for the
potential~$- t \log |D\pF| - p m_f$; 
we denote it by~$\psP(t, p)$ (resp. $\psP^{\R}(t, p)$).
On the set where it is finite, the function~$\psP$ (resp.~$\psP^{\R}$) so
defined is strictly decreasing in each of its variables.

Given~$t > 0$ and~$p$ in~$\R$, a finite measure~$\tmu$ on~$\C$ that is supported on the maximal invariant set of~$F_{f}|_{D_f \cap \R}$ (resp.~$F_f$) is \emph{$(t, p)$\nobreakdash-conformal for~$F_f$}, if for every~$W$ in~$\fD^{\R}_f$ (resp. $\fD_f$), and every Borel subset~$U$ of~$W \cap \R$ (resp.~$W$), we have
\begin{displaymath}
\tmu(F_f(U))
=
\exp(p m_f(W)) \int_U |DF_f|^t \dd \tmu.
\end{displaymath}
Note that in this case we have
\begin{multline}
  \label{eq:6}
    \exp(- p m_f(W)) \inf_{z \in W} |DF_f (z) |^{-t}
  \le
  \tmu(W)
  \\ \le
  \exp(- p m_f(W)) \sup_{z \in W} |DF_f (z) |^{-t}.
\end{multline}
\begin{prop}[\cite{CorRiv1708}, Proposition~5.2]
  \label{p:equilibria}
  Let~$n \ge 5$ be an integer, $f$ in~$\cK_n(\sF)$, and~$t > 0$ such that
\begin{equation}
  \label{e:Bowen formula}
\sP^{\R}_f(t, P_f^{\R}(t)) = 0
\left( \text{resp. }
\sP_f(t, P_f(t)) = 0 \right).
\end{equation}
Then there is a~$(t, P^\R_f(t))$\nobreakdash-conformal (resp.~$(t, P_f(t))$\nobreakdash-conformal) probability measure~$\tmu$ for~$F_f$, and there is a probability measure~$\trho$ that is invariant by~$F_f$, absolutely continuous with respect to~$\tmu$, and whose density satisfies
\begin{equation}
  \label{e:density bound}
  \Xi_1^{-t} \le \frac{\dd \trho}{\dd \tmu} \le \Xi_1^t.
\end{equation}
If in addition
\begin{multline}
  \label{e:time integrability}
\sum_{W \in \fD^\R} m_{f}(W) \cdot \exp( - m_{f}(W) P^\R_f(t)) \sup_{w \in W \cap \R} |DF_f(w)|^{-t}
\\ \left( \text{resp. }
\sum_{W \in \fD} m_{f}(W) \cdot \exp( - m_{f}(W) P_f(t)) \sup_{w \in W} |DF_f(w)|^{-t}
\right)
\end{multline}
is finite, then the measure
$$ \hrho \= \sum_{W \in \fD^{\R}} \sum_{j = 0}^{m_{f}(W) - 1} (f^j)_* \left( \trho|_{W \cap \R} \right)
\left( \text{resp. }
\sum_{W \in \fD} \sum_{j = 0}^{m_{f}(W) - 1} (f^j)_* \left( \trho_t|_W \right)
\right) $$
is finite and the probability measure proportional to~$\hrho$ is the unique equilibrium state of~$f|_{I(f)}$ (resp.~$f|_{J(f)}$) for the potential~$-t \log |Df|$. 
\end{prop}

\subsection{Statements of main results}
\label{ss:main results}
In this subsection we state Theorems~\ref{t:pressure estimates} and~\ref{t:temperature dependence}.
These results are based on the notion of ``admissible'' family of quadratic-like maps introduced in~\cite{CorRiv1708}, which we recall now.
The proofs of both of these theorems are given in~\S\ref{s:proofs}.

Given a normalized quadratic-like map~$f$, for each periodic point~$p$ of~$f$ of period~$m$ in~$\N$, put
$$
\chi_f(p) \= \frac{1}{m} \log|Df^m(p)|.
$$

\begin{defi}[Admissible family of quadratic-like maps]
  \label{d:admissibility}
A uniform family of quadratic-like maps~$\sF$ is \emph{admissible}, if for every sufficiently large integer~$n \ge 6$ the following properties hold.
\begin{enumerate}
\item[1.]
If we endow~$\sF$ with the topology of locally uniform convergence, then there is a continuous function~$s_n \colon \cK_n \to \cK_n(\sF)$ such that~$c \circ s_n$ is the identity.
\item[2.]
  For every~$f$ in~$s_n(\cK_n)$, we have
  \begin{equation}
    \label{eq:5}
      \chi_f(p(f))
      >
      \chi_f(p^+(f))
      \text{ and }
      \chi_f(p^+(f))
    =
    \chi_f(p^-(f)).
    \end{equation}  
\end{enumerate}
\end{defi}

Given a uniform family~$\sF$, for each integer~$n \ge 5$, each~$f$ in~$\cK_n(\sF)$ put
\begin{equation}
\label{d:theta n}
\theta(f)
\=
\left| \frac{Dg_f(p(f))}{Dg_f(\wtp(f))} \right|^{1/2}
\text{ and }
t_*(f) \= \frac{ \log 2}{\log \theta(f)}.
\end{equation}
Note that~$t_*(f)$ is a real analytic function of~$f$.

Endow the set~$\{+, - \}$ with the discrete topology, and~$\signs$ with the corresponding product topology.

\begin{theoalph}[Phase transitions]
  \label{t:pressure estimates}
For every~$R_0 > 0$ there is a constant~$K_0 > 1$ such that if~$\sF_0$ is an admissible uniform family of quadratic-like maps with constants~$K_0$ and~$R_0$, then for every sufficiently large integer~$n$ there is a continuous subfamily~$(\fs)_{\uvarsigma \in \signs}$ of~$s_n(\cK_n)$ such that the following properties hold.

For each~$\uvarsigma$ in~$\signs$ the map~$\fs$ is essentially topologically exact and we have $\chi_{\rm{crit}}(\fs) > 0$.
Moreover, for each positive~$t  < t_*(\fs)$ there is a unique equilibrium state~$\rho_{t}^{\R}(\uvarsigma)$ (resp. $\rho_{t}(\uvarsigma)$) of~$\fs|_{I(\fs)}$ (resp.~$\fs|_{J(\fs)}$) for the potential $- t \log |D\fs|$. 

Furthermore, there are positive numbers~$q$, $\xi$,  $\Xi> 2\xi$, and $\Delta\ge 1$ such that for every~$\uvarsigma$ in~$\signs$ the following property holds. 
Put 
$$
t_*\= t_*(\fs)
\text{ and }
t_0
\=
\left(1 - \frac{1}{400q} \right) \cdot t_*
$$
and define the functions~$\delta^+$, $\delta^-$, $P^+$,
$P^- : (t_0, + \infty) \to \R$, by
\[
  \delta^+(t)
  \= 
\begin{cases}
\frac{2 \log 2}{3} \cdot 2^{- q \left(\left( \frac{t_*(\Xi-2\xi)}{q\left( t_* - t \right)}\right)^{\frac{1}{2}} - 1 \right)^3}
& \text{if } t \in (t_0,t_*);
\\
0 & \text{if } t \ge t_*;
\end{cases}
\]
\[
 \delta^-(t)
\=
\begin{cases}
\frac{\log 2}{3} \cdot 2^{- q \left(\left( \frac{t_*(\Xi+2\xi)}{q\left( t_* - t \right)}\right)^{\frac{1}{2}} +\Delta \right)^3}
& \text{if } t \in (t_0,t_*);
\\
0 & \text{if } t\ge t_*;
\end{cases}    
\]
\[
P^+(t)\= - t \frac{\chi_{\rm{crit}}(\fs)}{2} + \delta^+(t),
\text{ and }
P^-(t)\= - t \frac{\chi_{\rm{crit}}(\fs)}{2} + \delta^-(t).
\]
Then for~$t >t_0$ we have
$$
P^-(t)  \le P_{\fs}^\R(t) \le  P_{\fs}(t) \le P^+(t),
$$
and for $t \ge t_*$ there is no equilibrium state of~$\fs|_{I(\fs)}$
(resp.~$\fs|_{J(\fs)}$) for the potential~$- t \log |D\fs|$  and we have
$$
\sP_{\fs}^\R\left(t, - t \frac{\pchicrit}{2}\right)
\le 
\sP_{\fs}\left(t, - t\frac{\pchicrit}{2}\right)
<
0. $$
Finally, for~$t \ge t_*$ and for~$p$ in~$\R$ the following properties hold:
\begin{enumerate}
\item[1.]
If~$p \ge - t \pchicrit/2$, then there is a unique $(t, p)$\nobreakdash-conformal probability measure for~$\fs|_{I(\fs)}$ (resp.~$\fs$) supported on~$I(\fs)$ (resp.~$J(\fs)$).
Moreover, this measure is dissipative, purely atomic, and supported on the backward orbit of~$z = 0$.
\item[2.]
If~$p < - t \pchicrit/2$, then there is no $(t, p)$\nobreakdash-conformal probability measure for~$\fs|_{I(\fs)}$ (resp.~$\fs$) supported on~$I(\fs)$ (resp.~$J(\fs)$).
\end{enumerate}
\end{theoalph}

For proving the High-Order Phase Transitions Theorem, observe that $t_*(\fs)$ is continuous in the compact set $\signs$.

\begin{theoalph}[Sensitive dependence]
\label{t:temperature dependence}
For every~$R_0 > 0$ there is a constant~$K'_0 > 1$ such that if~$\sF_0$ is an admissible uniform family of quadratic-like maps with constants~$K'_0$ and~$R_0$, then for every sufficiently large integer~$n$ there is a continuous subfamily~$(\fs)_{\uvarsigma \in \signs}$ of~$s_n(\cK_n)$ satisfying the conclusions of Theorem~\ref{t:pressure estimates} and such that the following property holds:

There are constants~$C > 0$ and~$\varkappa > 0$, and a strictly increasing sequence  of positive real numbers $(\tau_m)_{m\in \N}$ converging to~1 such that for every sequence~$\uvarsigma = (\varsigma(m))_{m \in \N}$ in~$\signs$ the following properties hold.
Put~$t_* \= t_*(\fs)$, and let~$m$ and~$\whm$ be integers such that
$$ \whm \ge m \ge 1
\quad \text{and} \quad
\varsigma(m) = \cdots = \varsigma(\whm) $$
and let~$t$ be in~$[t_* \tau_m, t_* \tau_{\whm}]$.
Then the equilibrium state~$\rho_{t}^{\R}(\uvarsigma)$ (resp.~$\rho_{t}(\uvarsigma)$) of~$\fs|_{I(\fs)}$ (resp.~$\fs|_{J(\fs)}$) satisfies
\begin{multline*}
\rho_t^{\R}(\uvarsigma) \left( B \left( \cO^{\varsigma(m)}(\fs), \exp\left(\frac{- \varkappa}{t_* - t}\right) \right) \right)
\ge
1 - C \exp\left(\frac{- \varkappa}{t_* - t}\right)
\\
\left( \text{resp. } 
 \rho_t(\uvarsigma) \left( B \left( \cO^{\varsigma(m)}(\fs), \exp\left(\frac{- \varkappa}{t_* - t}\right) \right) \right)
\ge
1 - C \exp\left(\frac{- \varkappa}{t_* - t}\right)
\right).
\end{multline*}
\end{theoalph}

Our estimates show that for every $\uvarsigma$ in $\signs$ every accumulation measure of $(\rho_t^{\R}(\uvarsigma))_{t>0}$ and of $(\rho_t(\uvarsigma))_{t>0}$ is supported on $\cO^{+}(\fs)\cup \cO^{-}(\fs)$, see Remark~\ref{r:accumulation}.

\subsection{Proof of the Sensitive Dependence at Positive Temperature assuming Theorem~\ref{t:temperature dependence}}
\label{ss:proof of positive sensitive dependence}
Let~$K'_0$ be the constant given by Theorem~\ref{t:temperature dependence} for $R_0=80$.
Let~$r_\#$ be the constant given by~\cite[Lemma~3.5]{CorRiv1708}  for this choice of~$K'_0$. Put $\Lambda_0 = B(0,r_\#)$, and
let $\sF_0 = (\df)_{\lambda\in \Lambda_0}$ be the family of quadratic-like maps constructed in \cite[Lemma~3.5]{CorRiv1708}  for this choice of~$K'_0$. 
By  \cite[Lemma~3.5]{CorRiv1708} and~\cite[(3.3)]{CorRiv1708} the family $\sF_0$ is uniform with constants $K'_0$ and $80$, and admissible.
Fix a sufficiently large integer~$n$ for which the conclusions of Theorem~\ref{t:temperature dependence} are satisfied, and let~$(\fs)_{\uvarsigma \in \signs}$ and $(\tau_m)_{m \in \N}$ as in the statement of Theorem~\ref{t:temperature dependence}.
Given~$\uvarsigma$ in~$\signs$, denote by~$\lambda(\uvarsigma)$ the unique parameter in~$\Lambda_0$ such that~$\whf_{\lambda(\uvarsigma)} = \fs$.
By~\cite[Lemma~3.5]{CorRiv1708}, $\lambda(\uvarsigma)$ is real.
We prove the Sensitive Dependence at Positive Temperature with~$\Lambda = \{ \lambda(\uvarsigma) \mid \uvarsigma \in \signs \}$.

Let~$\lambda_0$ in~$\Lambda$ be given, and let~$(\delta_\ell)_{\ell\in \N}$ be an increasing sequence of positive numbers converging to~1 as $\ell\to +\infty$.
Let~$\uvarsigma_0$ in $\signs$ be such that~$\lambda_0\=\lambda(\uvarsigma_0)$, and put~$t_{*, 0} \= t_*(f_{\uvarsigma_0})$.
Replacing~$( \delta_{\ell} )_{\ell \in\N}$ by a subsequence if necessary, assume that for every $\ell\in \N$ there is a unique $m\in \N$ such that 
$$
\delta_\ell\in [\tau_m,\tau_{m+1})\text{ and } 
\delta_{\ell+1}\ge \tau_{m+2}.$$
Put~$m(\ell) \= m$.
We see that $m(\ell+1)\ge m(\ell)+2$, which implies
\begin{equation}
\label{e:delta}
\delta_\ell\in [\tau_{m(\ell)}, \tau_{m(\ell+1)-1}].
\end{equation}
Let~$M$ be a given positive integer. Let $\ell_M$ be the smallest $\ell$ such that $m(\ell)\ge M$.
Let~$\uvarsigma = (\varsigma(m))_{m \in \N}$ be a sequence in~$\signs$ that agrees with~$\uvarsigma_0$ up to $m(\ell_M)$ and    such that   for every even (resp. odd) integer~$\ell > \ell_M$ and every~$m$ in~$[m(\ell), m(\ell+1) - 1]$ we have~$\varsigma(m) = +$ (resp. $\varsigma(m) = -$).

Put $\lambda\=\lambda(\uvarsigma)$ and $t_* \= t_*(f_{\uvarsigma})$.
Since the subfamily $(\fs)_{\uvarsigma \in \signs}$ is continuous, taking $M$ sufficiently large we can ensure that $\lambda$ is arbitrarily close to $\lambda_0$ and that~$t_*$ is arbitrarily close to $t_{*, 0}$. 

By~\eqref{e:delta} we have that~$t_*(\uvarsigma)\cdot \delta_\ell$ is in~$[t_*(\uvarsigma)\cdot\tau_{m(\ell)}, t_*(\uvarsigma)\cdot\tau_{m(\ell+1)-1}]$.
Since by definition of~$\uvarsigma$ for each~$\ell$ in~$\N$  larger than $\ell_M$ we have
$$ \varsigma(m(\ell)) = \cdots = \varsigma(m(\ell + 1) - 1), $$
and since~$\varsigma(m(\ell))$ alternates between~$+$ and~$-$ according to whether~$\ell$ is even or odd, Theorem~\ref{t:temperature dependence} implies the desired assertion for this choice of~$\uvarsigma$ 
and thus, for~$\lambda$ and~$t_*$.

\section{Proof of main results}
\label{s:proofs}
In this section we prove Theorems~\ref{t:pressure estimates} and~\ref{t:temperature dependence} in~\S\ref{ss:proof of pressure estimates} and~\S\ref{ss:proof of temperature dependence}, respectively.
Indeed, for both theorems we use the same family of maps. 
More precisely,   put~$\upsilon \= \frac{1}{4} \log 2$, and let~$R_0 > 0$ be given. Let~$K_1$, $n_1$, and~$\kappa_1$ be given by Proposition~\ref{p:transporting Peierls} with~$R' = R_0$.
We prove that Theorems~\ref{t:pressure estimates} and~\ref{t:temperature dependence} hold with~$K_0 =  K_1$.
Let~$\sF_0$ be a uniform family of quadratic-like maps that admits~$K_0$ and~$R_0$ as uniformity constants, and that is admissible.
By Proposition~\ref{p:transporting Peierls} with~$\sF_1 = \sF_0$, for every~$n \ge n_1$, every~$f$ in~$\cK_n(\sF_0)$ satisfies the Geometric Peierls Condition with constants~$\kappa_1$ and~$\upsilon$, and by~\eqref{e:choice of n for negative pressure} we have
\begin{equation}
  \label{e:large phase transition}
\log \theta(f) \le \frac{3}{8}\upsilon.
\end{equation}
Taking~$n_1$  larger if necessary, assume that for every~$n \ge n_1$ there is a continuous function~$s_n \colon \cK_n \to \cK_n(\sF_0)$ such that~$c \circ s_n$ is the identity, and that~\eqref{eq:5} holds for every~$f$ in~$s_n(\cK_n)$.

Let~$\Xi_1$, $\Xi_2$, $C_0$,$\upsilon_0$ and ~$\Xi_3$ be the constants given by Lemmas~\ref{l:bounded distortion to nice}, \ref{l:distortion to central 0},  \ref{l:contractions}, and \ref{l:2 variables functions},  respectively, with~$\sF = \sF_0$.
Moreover, let~$n_2$ and~$C_1$ the constants given by Proposition~\ref{p:improved MS criterion} with~$\sF = \sF_0$ and~$\kappa = \kappa_1$, let~$n_3$ and~$C_2$ be given by Proposition~\ref{p:estimating Z_1 by the postcritical series} with~$\sF = \sF_0$, and let~$n_{\&} \ge \max \{n_1,n_2, n_3\}$ be sufficiently large so that
\begin{equation}
  \label{e:choice of n I}
\exp(n_{\&} \upsilon) \ge \Xi_3^{\frac{1}{2}} C_1 \left(25 + 5\cdot \Xi_2 \right).
\end{equation}

\subsection{The subfamily}
\label{ss:subfamily}
In this subsection we define the subfamily $(f_{\uvarsigma})_{\uvarsigma\in \{+,-\}^{\N}}$, as in the statement of Theorems~\ref{t:pressure estimates} and~\ref{t:temperature dependence}.
Fix an integer~$n \ge n_{\&}$, let~$c_{\&}$ in~$\cK_n$ be such that~$\iota(c_{\&})$ is the constant sequence equal to~$0$, and put~$f_{\&} \= s_n(c_{\&})$.
By~\eqref{eq:5} we have~$\theta(f_{\&}) > 1$, so there is~$r_{\&} > 0$ such that for~$c$ in~$B(c_{\&}, r_{\&}) \cap \cK_n$ the number~$\theta(s_n(c))$ is defined, and depends continuously on~$c$.
Reducing~$r_{\&}$ if necessary, assume that for all~$c$ and~$c'$ in~$B(c_{\&}, r_{\&}) \cap \cK_n$ we have~$\theta(s_n(c)) \le \theta(s_n(c'))^2$.
By Proposition~\ref{p:ps} it follows that there is an integer~$q_{\&} \ge 0$ such that the set
$$ \{ c \in \cK_n \mid \text{ for every~$j$ in~$\{0, \ldots, q_{\&} \}$, $\iota(c)_j = 0$} \} $$
is a compact set contained in~$B(c_{\&}, r_{\&})$.
Put 
$$
t_{*,\sup} \=  \sup_{c \in \cK_n \cap B(c_{\&}, r_{\& })} \frac{\log2}{\log \theta(s_n(c)) }, 
t_{*,\inf} \=  \inf_{c \in \cK_n \cap B(c_{\&}, r_{\& })} \frac{\log2}{\log \theta(s_n(c)) }, 
$$
$$
\eta_0\=  \sup_{f \in s_n( \cK_n \cap B(c_{\&}, r_{\& }))} \frac{\exp(\chi_{f}(\beta(f)))}{\exp(\chi_\text{crit}(f))},
$$
and fix $\Omega > 0$ sufficiently large such that 
\begin{equation}\label{e:choice of Omega}
2^{ 2\Omega-3 } 
 \ge
2^n \left( \Xi_3^{\frac{1}{2}} C_1 \eta_0^{\frac{n}{2}}\right)^{t_{*,\sup}}.
\end{equation}
On the other hand, by~\eqref{eq:5} for each~$c$ in~$\cK_n$ the number
$$
\xi(s_n(c))=\frac{\log \Xi_2}{2\log\theta(s_n(c))},
$$ 
defined as in~\S\ref{ss:estimating postcritical series},
depends continuously on~$c$ in~$\cK_n$, so
$$ \xi \= \sup_{c \in \cK_n \cap B(c_{\&}, r_{\& })} \xi(s_n(c)) < + \infty. $$
Note that by our choice of $r_\&$ for every $c \in \cK_n \cap B(c_{\&}, r_{\& })$ we have 
\begin{equation}
\label{e:supxi}
\xi\le 2\xi(s_n(c)).
\end{equation}
Put~$\Xi \= \lceil 2\xi \rceil + 1$ as in~\S\ref{ss:positive temperature partition}, and let~$q \ge q_{\&}$ be a sufficiently large integer satisfying the conditions in~\S\ref{ss:positive temperature partition} and such that in addition $q+\Xi$ is even.
For each real number~$s \ge 0$, let~$I_s$ and~$J_s$ be the intervals defined in~\S\ref{ss:positive temperature partition} for these choices of~$\Xi$ and~$q$. 
By definition, as~$s$ varies in~$\N_0$ these intervals form a partition of~$[1, +\infty)$.
Moreover, for each integer~$s$ in~$\N_0$, the end points of~$I_s$ and~$J_s$ are even.

Endow the set~$\{0, 1^+, 1^- \}$ with the discrete topology, and $\{0, 1^+, 1^- \}^{\N_0}$ with the corresponding product topology.
Moreover, endow the subset~$\hSigma$ of~$\{0, 1^+, 1^- \}^{\N_0}$, defined in~\S\ref{ss:estimating postcritical series}, with the induced topology.

As in~\S\ref{ss:estimating the 2 variables series}, for every $\tau \in (0,1)$, put
$$ s^+(\tau) \= \left(\frac{\Xi - 2\xi}{q(1 - \tau)}\right)^{\frac{1}{2}} \text{ and } s^-(\tau) \= \left(\frac{\Xi + 2\xi}{q(1 - \tau)}\right)^{\frac{1}{2}}.$$
Let $(\tau_m)_{m\in \N_0}$ be an increasing sequence in $[1/2,1)$  converging to~1 satisfying the following properties:
\begin{align}
\label{e:tau 0}& \tau_0 \ge \left(1 - \frac{1}{400q} \right); \\
\label{e:tau 0 2}& t_{*,\sup}(1-\tau_0)\log C_0 \le t_{*,\inf} \frac{\upsilon_0}{8q}; \\
\label{e:tau recursion}& \text{For every } m \in \N, \text{ we have } s^+(\tau_{m}) \in \N, \text{ and } s^-(\tau_{m-1})+\Omega  + 6 < s^+(\tau_{m}).
\end{align}	
Notice that  the inequality $\Xi-2\xi\ge 1$ implies that for every $\tau \ge 1 - \frac{1}{400q}$ we have
\begin{equation}
\label{e:lower bound s+}
s^+(\tau)\ge 20.
\end{equation}
Put $\ell(0)=0$ and for every  $m$ in $\N$  put  $\ell(m) \= s^+(\tau_{m}) - 5$.
Given $\uvarsigma$ in $\{+,-\}^{\N}$, let~$\uwhx(\uvarsigma)$ be the sequence in~$\{0, 1^+, 1^- \}^{\N_0}$ defined by
$$ \uwhx(\uvarsigma)_j
\=
\begin{cases}
0
&
\text{if for some~$s$ in~$\N_0$ we have $j+1 \in I_s$};
\\
1^+
&
\text{if $j+1 \in \bigcup_{k\in [\ell(0),\ell(1))} J_{k}$};
\\
1^{\varsigma(m)}
&
\text{for } j +1 \in \bigcup_{k\in [\ell(m),\ell(m+1))} J_{k}.
\end{cases}
$$
Note that by definition~$\uwhx(\uvarsigma)$ is in~$\hSigma$ and that the first~$q$ entries of this sequence are equal to~$0$.
Moreover, the map~$\uwhx : \signs \to \hSigma$ so  defined  is continuous.
Finally, note that the length of each maximal block of~$1^-$'s in~$\uwhx(\uvarsigma)$ is even.

Define the family of itineraries~$(\iota(\uvarsigma))_{\uvarsigma \in \signs}$ in~$\{0, 1\}^{\N_0}$, by
$$ \iota(\uvarsigma)_j
=
\begin{cases}
0
&
\text{if $\uwhx(\uvarsigma)_j = 0$};
\\
1
&
\text{if~$\uwhx(\uvarsigma)_j = 1^+$};
\\
0
&
\text{if~$\uwhx(\uvarsigma)_j = 1^-$ and~$j$ is even};
\\
1
&
\text{if~$\uwhx(\uvarsigma)_j = 1^-$ and~$j$ is odd}.
\end{cases}
$$
Note that~$\iota(\uvarsigma)$ is compatible with~$\uwhx(\uvarsigma)$ in the sense of~\S\ref{ss:estimating postcritical series} and that the first~$q$ entries of this sequence are equal to~$0$. Moreover, $\iota(\uvarsigma)$ depends continuously on~$\uvarsigma$ in~$\signs$.

Given~$\uvarsigma$ in~$\signs$, let~$c(\uvarsigma)$ in~$\cK_n$ be the unique parameter such that~$\iota(f_{c(\uvarsigma)}) = \iota(\uvarsigma)$ (Proposition~\ref{p:ps}), and put~$\fs \= s_n(c(\uvarsigma))$.
Note that the function~$\uvarsigma \mapsto \fs$ so defined is continuous.
On the other hand, since for each~$\uvarsigma$ in~$\signs$ the first~$q \ge q_{\&}$ entries of~$\iota(\fs) = \iota(\uvarsigma)$ are equal to~$0$, the parameter~$c(\uvarsigma)$ is in~$B(\lambda_{\&}, r_{\&})$.
So, for every~$\uvarsigma$ and~$\uvarsigma'$ in~$\signs$ we have
\begin{equation}
  \label{e:theta distortion}
\theta(\fs) \le \theta(f_{\uvarsigma'})^2.
\end{equation}

This finishes the definition of the continuous subfamily~$(\fs)_{\uvarsigma \in \signs}$ of~$s_n(\cK_n)$.
 
\subsection{Proof of Theorem~\ref{t:pressure estimates}}
\label{ss:proof of pressure estimates}
The purpose of this subsection is to finish the proof of Theorem~\ref{t:pressure estimates}.
That for each~$\uvarsigma$ in~$\signs$ the map~$\fs|_{I(\fs)}$ is essentially topologically exact follows from the fact that this map is not renormalizable, see~\cite[\S3]{CorRiv13} for details.
To prove that for each~$\uvarsigma$ in~$\signs$ there is a unique equilibrium  state for every positive~$t<t_*(\fs)$ we just need to show that  
$P_{\fs}^{\R} (t) > -t \cdot \chi_{\text{crit}}(\fs)/2$.   
This follows from the estimates for the pressure function given below.
These estimates will be also used in the proof of Theorem~\ref{t:temperature dependence} in the next section, and they depend only on Lemma~\ref{l:positive temperature first floor}.

Let~$N : \N_0 \to \N_0$ and~$B : \N_0 \to \N_0$ be the functions defined in~\S\ref{ss:positive temperature partition} for our choices of~$\Xi$ and~$q$.
Clearly, $N(k)/k \to 0$ as~$k \to + \infty$, and for each~$\uvarsigma$ in~$\signs$, these functions coincide with those defined in~\S\ref{ss:estimating postcritical series} with~$(\whx_j)_{j \in \N_0} = \uwhx(\uvarsigma)$.
It follows that for each integer~$k \ge 0$, the~2 variables functions~$\pi_{\fs,k}^{\pm}$ defined in~\S\ref{ss:estimating postcritical series} 
for our choice of~$\xi$ in~\S\ref{ss:subfamily} are independent of~$\uvarsigma$; we denote them by~$\pi_k^{\pm}$.
Note that the 2~variables series defined in~\S\ref{ss:positive temperature partition}: $\Pi^{\pm}$,
and for each integer~$s \ge 0$, the series~$I_s^{\pm}$, and~$J_s^{\pm}$, satisfy
$$ \Pi^{\pm} = \sum_{k = 0}^{+ \infty} \pi_k^{\pm},
I_s^{\pm} = \sum_{k \in I_s} \pi_k^{\pm},\text{ and }
J_s^{\pm} = \sum_{k \in J_s} \pi_k^{\pm}.$$
Finally, for each real number~$s \ge 0$ put~$\lambda(s) = |J_s|^{-1}$, as in~\S\ref{ss:positive temperature partition}.

Fix~$\uvarsigma$ in~$\signs$ and put
$$ f\= \fs,
\wtp\= \wtp(\uvarsigma),
\whp \= \whp(\uvarsigma),
$$
$$ P^{\R} \= P_{f_{\uvarsigma}}^{\R},
\sP^{\R} \= \sP_{f_{\uvarsigma}}^{\R},
P\= P_{f_{\uvarsigma}},
\sP \= \sP_{f_{\uvarsigma}}, $$
$$ t_*\=t_*(\fs),  \text{ and } t_0\= \left(1 - \frac{1}{400q} \right) \cdot t_*.$$

Now we prove the estimates for the pressure functions for every $t$ in $\left(t_0,t_*\right)$.
Fix~$t$ in $\left(t_0,t_*\right)$ and  put
$$ \tau
\=
\frac{1 }{t_*} t
=
\frac{\log \theta(f)}{\log 2} t.$$
By~\eqref{e:lower bound s+},  we have $s^+(\tau) - 2 \ge 18.$ Put
$$ \hP^+
\=
- t \frac{\pchicrit}{2} + \frac{\log 2}{3} \lambda(s^+(\tau) - 2),
\quad \text{and} \quad
\hP^-
\=
- t \frac{\pchicrit}{2} + \frac{\log 2}{3} \lambda(s^-(\tau)+\Omega). $$
By~\eqref{e:large phase transition} we have that
\begin{equation}
\label{e:t_0 condition}
t_0 \ge 8 = \frac{2\log2}{\upsilon}. 
\end{equation}

First we prove the lower bound for the real pressure function. 
By~\eqref{eq:5}, \eqref{e:lambda upper bound}, Lemma~\ref{l:2 variables functions} with $\uwhx=\uwhx(\uvarsigma)$ and~$\delta = \frac{\log 2}{3} \lambda(s^-(\tau)+\Omega) \le \log 2$,  Lemma~\ref{l:positive temperature first floor}, the definition of $\eta_0$,  and
\eqref{e:choice of Omega}  we have
\begin{multline}
\label{e:postcritical series lower bound}
\sum_{k = 0}^{+ \infty} \exp \left(-(n+3k) \hP^- \right) |Df^{n+3k}(f(0))|^{-\frac{t}{2}}
\\
\begin{aligned}
& \ge \Xi_3^{- \frac{t}{2}} \exp \left( - n \frac{\log 2}{3} \lambda(s^-(\tau)+\Omega) \right) \left( \frac{\exp (\pchicrit)}{|Df(\beta(f))|}
\right)^{\frac{t}{2} n }
\Pi^- \left( \tau, \lambda(s^-(\tau)+\Omega) \right).
\\ & \ge
2^{-n} \left( \Delta^{\frac{1}{2}} \eta_0^{\frac{n}{2}} \right)^{-t} 2^{2\Omega-3}
\\ & 
\ge
C_1^t.
\end{aligned}
\end{multline}

By~\eqref{e:choice of n for negative pressure} we have~$ 2\upsilon \le \pchicrit $.
Thus, using that $\lambda(s)$ is strictly decreasing, by part~1 of Lemma~\ref{l:the itinerary}, and~\eqref{e:lambda upper bound}, we deduce 
$$\hP^-<\hP^+ < 0.$$
Together with~\eqref{e:t_0 condition} and using part~1 of Proposition~\ref{p:improved MS criterion}, we get that
\begin{equation}
  \label{e:pressure lower bound}
P(t) \ge P^{\R}(t) \ge \hP^-
\quad \text{and} \quad
\sP(t, \hP^-) \ge \sP^{\R}(t, \hP^-) > 0.  
\end{equation}
Putting $\Delta\= \Omega+1$, we get $P^{\R}(t) \ge \hP^-\ge P^-(t)$, where $P^-(t)$ as in the statement of Theorem~\ref{t:pressure estimates}.

Now we prove the upper bound for the complex pressure function.
Notice that by~\eqref{e:supxi}  and~\eqref{e:theta distortion} we have
$$
2^{\tau \xi} = \theta(f)^{t\xi}\le \theta(f)^{2t\xi(f)} = \Xi_2^t. 
$$
Combined with~\eqref{eq:5}, Lemma~\ref{l:2 variables functions} with $\uwhx=\uwhx(\uvarsigma)$ and~$\delta = \frac{\log 2}{3} \lambda(s^+(\tau) - 2)$, Lemma~\ref{l:positive temperature first floor},  and~\eqref{e:choice of n I}, we obtain that
\begin{multline*}
\label{e:postcritical series upper bound}
\sum_{k = 0}^{+ \infty} \exp \left(-(n+3k) \hP^+ \right) |Df^{n+3k}(f(0))|^{-\frac{t}{2}}
\\
\begin{aligned}
& \le \Xi_3^{\frac{1}{2}} \left( \frac{\exp (\pchicrit)}{|Df(\beta(f))|} \right)^{\frac{t}{2} n }
\Pi^+ \left( \tau, \lambda(s^+(\tau) - 2) \right).
\\ & \le
\left( \Xi_3^{\frac{1}{2}} \exp(- n \upsilon) \right)^t (25 +
5 \cdot 2^{\tau \xi})
\\ & \le
\left( \Xi_3^{\frac{1}{2}} \exp(- n \upsilon)  \left( 25 +5\cdot \Xi_2 \right) \right)^t
\\ & \le
C_1^{-t}.
\end{aligned}
\end{multline*}
Together with part~2 of Proposition~\ref{p:improved MS criterion}, this implies
\begin{equation}
  \label{e:pressure upper bound}
P^{\R}(t) \le P(t) \le \hP^+
\quad \text{and} \quad
\sP^{\R}(t, \hP^+) \le \sP(t, \hP^+) < 0.  
\end{equation}
To conclude we observe that by part~2 of Lemma~\ref{l:the itinerary} we have that $ \hP^+\le P^+(t)$, 
where $P^+(t)$ is as in the statement of Theorem~\ref{t:pressure estimates}.
This finish the prove of the estimates for the pressure functions for~$t$ in $(t_0,t_*)$.

Now we prove the existence and uniqueness of equilibrium states for every~$t$ in~$(0,  t_*)$. Put
$$ \chiinfR (f)
\=
\inf \left\{ \int \log |Df| \dd \mu \mid \mu \in \sM_f^{\R} \right\},
$$
and
$$ \chiinfC (f)
\=
\inf \left\{ \int \log |Df| \dd \mu \mid \mu \in \sM_f \right\}. $$
Combining~\eqref{e:pressure lower bound} and~\eqref{e:pressure upper bound}, we deduce that
$$ 
P(t_*)=P^{\R}(t_*) 
=  \lim_{t\nearrow t_*} P^{\R}(t) 
=
- t_* \frac{\chicritf}{2}. $$
Using that for every $t>0$ we have $P^{\R}(t) \ge -t\chiinfR (f)$, we deduce that $\chiinfR (f)\ge \chicritf/2$, and by \cite[Proposition~5.6]{CorRiv1708} we get $\chiinfR (f) = \chicritf/2$.
Similarly, $\chiinfC(f) = \chicritf /2$.
Using~\eqref{e:pressure lower bound} again, we conclude that for every~$t $ in $(0,t_*)$ we have
\begin{equation*}
P^{\R}(t) > - t\chiinfR (f)
\quad \text{and} \quad 
P(t) > - t \chiinfC (f).
\end{equation*}
The existence and uniqueness of equilibrium states follows from~\cite{PrzRiv1405} in the real case and from~\cite{PrzRiv11} in the complex case.

For the estimates for the pressure functions for $t\ge t_*$, observe that  
$$
P(t_*)=P^{\R}(t_*)=- t_* \chiinfR (f)
$$
and the definition of the pressure function imply that 
for every~$t $ in $[t_*,+\infty)$ we have
\begin{equation}
\label{e:pressure asintota}
P(t)=P^{\R}(t) 
=
- t\frac{\chicritf}{2}.
\end{equation}

Finally, we prove for $t\ge t_*$ the assertion about the $(t,p)$-conformal measures and the non-existence of equilibrium states.
Using~\eqref{eq:5}, Lemma~\ref{l:2 variables functions} with $\uwhx=\uwhx(\uvarsigma)$ and~$\delta = 0$, part~1 of Sublemma~\ref{subl:11},  and~\eqref{e:choice of n I}, this implies that for every $t$ in $\left[ t_*,+\infty\right)$,
\begin{multline*}
\label{e:postcritical series upper bound 1}
\sum_{k = 0}^{+ \infty} \exp \left(-(n+3k) \left(-t\pchicrit/2\right) \right) |Df^{n+3k}(f(0))|^{-\frac{t}{2}}
\\
\begin{aligned}
& \le \Xi_3^{\frac{1}{2}} \left( \frac{\exp (\pchicrit)}{|Df(\beta(f))|} \right)^{\frac{t}{2} n }
\Pi^+ \left( \tau,0) \right).
\\ & \le
\left( \Xi_3^{\frac{1}{2}} \exp(- n \upsilon) \right)^{t} 2( 2^{\tau \xi}+1)
\\ & \le
\left( \Xi_3^{\frac{1}{2}} \exp(- n \upsilon)  \left( 2 +2\cdot \Xi_2 \right) \right)^{t}
\\ & \le
C_1^{-t}.
\end{aligned}
\end{multline*}
Together with part~2 of Proposition~\ref{p:improved MS criterion}, this implies that for every $t$ in $\left[ t_*,+\infty\right)$,
\begin{equation*}
  \label{e:pressure upper bound 2}
\sP^{\R}\left(t, -t\frac{\pchicrit}{2}\right) \le \sP\left(t, -t\frac{\pchicrit}{2}\right) < 0.  
\end{equation*}
Thus, for every $p\ge -\pchicrit/2$ we get that 
$$
\sP^{\R}\left(t, p \right) \le \sP\left(t, p\right) \le \sP\left(t, -t\frac{\pchicrit}{2}\right)< 0. 
$$
Hence, the assertion about $(t,p)$-conformal measures follows from~\eqref{e:pressure asintota}, Proposition~\ref{p:conformal measures} and \cite[Lemma~5.5]{CorRiv1708}.

To prove the assertions about equilibrium states, let~$t \ge t_*$ be given and suppose by contradiction there is an 
equilibrium state~$\rho$ of~$f|_{I(f)}$ (resp.~$f|_{J(f)}$) for the potential~$- t \log |Df|$.
Since~$f$ satisfies the Collet-Eckmann condition, it follows that the Lyapunov exponent of~$\rho$ is strictly positive, see~\cite[Theorem~A]{NowSan98} or~\cite[Main Theorem]{Riv1204} for the real case and~\cite[Main Theorem]{PrzRivSmi03} for the complex case.
Then~\cite[Theorem~6]{Dob15b} in the real case and~\cite[Theorem~8]{Dob12} in the complex case imply that~$\rho$ is absolutely continuous with respect to the $(t, - t \pchicrit/2)$\nobreakdash-conformal measure for~$f|_{I(f)}$ (resp.~$f$) that is supported on~$I(f)$ (resp.~$J(f)$).
This implies in particular that~$\rho$ is supported on the backward orbit of~$z = 0$ and hence that~$\rho$ charges~$z = 0$.
This is impossible because this point is not periodic.
This contradiction shows that there is no equilibrium state of~$f|_{I(f)}$ (resp.~$f_{J(f)}$) for the potential~$- t \log |Df|$ and completes the proof Theorem~\ref{t:pressure estimates}.

\subsection{Proof of Theorem~\ref{t:temperature dependence}}
\label{ss:proof of temperature dependence}
In this subsection we complete the proof of Theorem~\ref{t:temperature dependence}.
We give the proof in the complex setting; except for the obvious notational changes, it applies to the real case without modifications.
The proof is based on the fact that for each~$\uvarsigma$ in~$\signs$ the equilibrium state~$\rho_t(\uvarsigma)$ is absolutely continuous with respect to  some  
$(t,P_{\fs}(t))$-conformal measure with a density bounded from above and from below (see Proposition~\ref{p:equilibria}). Then, we can estimate the measure via the derivative of the orbit  of the critical point. We keep the same notation as in the previous subsections.

Fix $t$ in $(t_0, t_*)$ and let~$m_0$ be in $\N$ such that $t$ is in~$[t_* \tau_{m_0-1}, t_* \tau_{m_0}]$.
Observe  that  $\tau_{m_0-1} \le  \tau \le \tau_{m_0}$.
On the other hand, note that by~\eqref{e:pressure lower bound} and~\eqref{e:pressure upper bound} there is~$s^{\C}$ in~$[s^+(\tau) - 2, s^-(\tau)+\Omega]$ such that
\begin{equation}
  \label{e:pressure control}
  P(t) = - t \frac{\chicritf}{2} + \frac{\log 2}{3} \lambda(s^{\C}).
\end{equation}

We start by proving that the hypotheses  of Proposition~\ref{p:equilibria} are satisfied for this value of $t$.
By~\eqref{eq:5},  Lemma~\ref{l:2 variables functions} with~$\delta = \frac{\log 2}{3} \lambda(s^-(\tau)+\Omega)$, and part~1 of Lemma~\ref{l:zero temperature tower}, we have
\begin{equation}
  \label{e:time postcritcal series}
\sum_{k = 0}^{+ \infty} k \cdot \exp \left(-(n+3k)\hP^- \right) |Df^{n+3k}(f(0))|^{-\frac{t}{2}}
<
+ \infty.  
\end{equation}
In particular, this implies that the sum in~\eqref{e:postcritical series lower bound} is finite, so by part~1 of Proposition~\ref{p:improved MS criterion} we have
$$
\sP(t, \hP^-)
<
+ \infty.$$
This implies that~$\sP(t, \cdot)$ is continuous and strictly decreasing on~$[\hP^-, + \infty)$, so by~\eqref{e:pressure lower bound}, \eqref{e:pressure upper bound}, and \cite[Proposition~II]{CorRiv1708} we have~\eqref{e:Bowen formula}.
Combining~\eqref{e:time postcritcal series} and  part~3 of Proposition~\ref{p:improved MS criterion}, we obtain
\eqref{e:time integrability}.
In view of~\eqref{e:pressure lower bound}, this implies the same sum in~\eqref{e:time integrability} with~$\hP^-$ replaced by~$P(t)$, is finite.
This complete the proof that the hypotheses of Proposition~\ref{p:equilibria} are satisfied.

Let $\trho$ and $\hrho$ be the measures given by Proposition~\ref{p:equilibria}. Put 
$\fD\= \fD_{f},
F \= F_{f}$
and for every integer~$k \ge 0$ put~$\fD_k \= \fD_{f, k}$.
For each integer~$s \ge 0$ let~$a_s$ and~$b_s$ be the left and right endpoint of~$I_s$, respectively as in~\S\ref{ss:positive temperature partition}.
Thus~$a_0 = 1$ and for every integer~$s \ge 0$ we have
$$ I_s = [a_s, b_s)
\quad \text{and} \quad
J_s = [b_s, a_{s + 1}). $$
Note that by part~2 of Lemma~\ref{l:the itinerary} and the hypothesis $q\ge 100(\Xi+1)$ in~\S\ref{ss:positive temperature partition} we have~$a_{s + 1} - b_s = |J_s| > (s + 1)^2$.

Put $s_0\= \lceil s^\C \rceil$ and for each integer~$\varsigma$ in~$[s_0 - 3, s_0]$ put
$$ \hrho_{\varsigma}'
\=
\sum_{k = b_{\varsigma} + {\varsigma}^2}^{a_{{\varsigma} + 1} - 1} \sum_{j = n + 3b_{\varsigma}-2}^{n + 3(k+1 - \varsigma^2)} \sum_{W \in \fD_{k}} (f^j)_* \left(\trho_t|_W \right), $$
and put
$$ \hrho''
\=
\sum_{k \in J_{s_0}} \sum_{j = n + 3b_{s_0 - 1}-2}^{n + 3(a_{s_0} +1  - s_0^2)} \sum_{W \in \fD_{ k}} (f^j)_* \left(\trho_t|_W \right). $$
Finally, put
$$ \hrho' \= \left( \sum_{\varsigma = s_0 - 3}^{s_0} \hrho_{\varsigma}' \right) +\hrho''. $$

In part~1 we estimate the total mass of~$\hrho'$ from below, and in part~2 we show that the total mass of~$\hrho_t - \hrho'$ is small in comparison to that of~$\hrho'$.
In part~3 we complete the proof of Theorem~\ref{t:temperature dependence} by showing  that the measure~$\hrho'$
is supported on a small neighborhood of the orbit of~$\wtp$ if~$\varsigma(m) = +$, and on a small neighborhood of the orbit of~$\whp$ if~$\varsigma(m) = -$.

Similarly to the previous subsection note that the~2 variables series defined in~\S\ref{ss:estimating the weighted 2 variables series}: $\tPi^{\pm}$, 
and for each integer~$s \ge 0$, the series~$\tI_s^+$, $\tJ_s^+$, and~$\hJ_s^\pm$, satisfy
$$ \tPi^+ = \sum_{k = 0}^{+ \infty} k \cdot \pi_k^+,
\tI_s^+ = \sum_{k \in I_s} k \cdot \pi_k^+,
\tJ_s^+ = \sum_{k \in J_s} k \cdot \pi_k^+ \text{ and } \hJ_s^{\pm} = \sum_{k = b_s + s^2}^{a_{s+1}-1} (k+1-b_s-s^2) \pi_k^{\pm} . $$

\partn{1}
To estimate the total mass of~$\hrho'$ from below  put~$\Upsilon_1 \= \Xi_1 C_2 \Xi_3^{\frac{1}{2}} \eta_0^{\frac{n}{2}} 2^n$.  Let~$\varsigma$ be an integer in~$[s_0 - 3, s_0]$.
By part~1 of Proposition~\ref{p:estimating Z_1 by the postcritical series}, Lemma~\ref{l:2 variables functions} with~$\delta = \frac{\log 2}{3}\lambda(s^{\C}) \le \log 2$, \eqref{eq:6}, \eqref{e:density bound}, and~\eqref{e:pressure control}, we have
\begin{displaymath}
  \begin{split}
|\hrho_{\varsigma}'|
& =
\sum_{k = b_{\varsigma} + {\varsigma}^2}^{a_{{\varsigma} + 1} - 1} 3(k + 1 - b_{\varsigma} - {\varsigma}^2) \sum_{W \in \fD_{ k}} \trho_t (W)
\\ & \ge
\left( \Xi_1 C_2 \right)^{-t} \sum_{k = b_{\varsigma} + {\varsigma}^2}^{a_{{\varsigma} + 1} - 1} 3(k + 1 - b_{\varsigma} - {\varsigma}^2) \exp(-(n + 3k)P(t)) |Df^{n + 3k}(f(0))|^{- t/2}
\\ & \ge
\left( \Xi_1 C_2 \Xi_3^{\frac{1}{2}} \left( \frac{|Df(\beta(\fs))|}{\exp(\chicritf)} \right)^{\frac{n}{2}}\right)^{-t} 2^{-n} \sum_{k = b_{\varsigma} + {\varsigma}^2}^{a_{{\varsigma} + 1} - 1} 3(k + 1 - b_{\varsigma} - {\varsigma}^2) \pi_k^-(\tau, \lambda(s^{\C}))
\\ & \ge
3 \Upsilon_1^{-t} \hJ_{\varsigma}^-(\tau, \lambda(s^{\C})).    
  \end{split}
\end{displaymath}
This implies
\begin{equation}
\label{e:central tower section I}
|\hrho'|
\ge
3\Upsilon_1^{-t} \sum_{\varsigma = s_0 - 3}^{s_0} \hJ_{\varsigma}^-(\tau, \lambda(s^{\C})).
\end{equation}

\partn{2} For~$\varsigma$ in~$[s_0 - 3, s_0]$, put
$$ H_{\varsigma}
\=
\{ k \in \N_0 \mid b_\varsigma + \varsigma^2 \le k \le a_{\varsigma + 1} - 1 \}. $$
By part~2 and the last inequality of Proposition~\ref{p:estimating Z_1 by the postcritical series}, \eqref{eq:6}, \eqref{e:density bound}, and~\eqref{e:pressure control}, we have
\begin{displaymath}
\begin{split}
\left| \hrho - \hrho' \right|
& =
\sum_{\substack{k \in \N_0 \\ k \not \in \bigcup_{\varsigma = \tau_0 - 3}^{s_0} H_{\varsigma}}} \sum_{W \in \fD_{ k}} m_{f}(W) \trho(W)
\\ & \quad +
\sum_{\varsigma = \tau_0 - 3}^{s_0 - 1} \sum_{k \in H_{\varsigma}} \sum_{W \in \fD_{ k}} \left(m_f(W) - 3 (k + 2 - b_{\varsigma} - \varsigma^2) \right) \trho(W)
\\ & \quad +
\sum_{k \in H_{s_0}} \sum_{W \in \fD_{ k}} \left(m_f(W) - 3 (k + 4 - b_{s_0} - 2s_0^2 + |J_{s_0 - 1}|))\right) \trho(W)
\\ & \le
(\Xi_1 C_2)^t \left[ \sum_{\substack{k \in \N_0 \\ k \not \in \bigcup_{\varsigma = \tau_0 - 3}^{s_0} H_{\varsigma}}} (n + 3k + 1) \exp(-(n + 3k) P(t)) |Df^{n + 3k}(f(0))|^{- \frac{t}{2}}
\right. \\ & \quad +
\sum_{\varsigma = \tau_0 - 3}^{s_0 - 1} \sum_{k \in H_{\varsigma}} (n + 3(b_{\varsigma} + \varsigma^2)) \exp(-(n + 3k) P(t)) |Df^{n + 3k}(f(0))|^{- \frac{t}{2}}
\\ & \quad + \left.
\sum_{k \in H_{s_0}} (n + 3(b_{s_0} + 2s_0^2 - |J_{s_0 - 1}|)) \exp(-(n + 3k) P(t)) |Df^{n + 3k}(f(0))|^{- \frac{t}{2}}
\right].    
  \end{split}
\end{displaymath}
Thus, if we put~$\Upsilon_2 \= \Xi_1 C_2 \Xi_3^{\frac{1}{2}} \exp(-n \upsilon)$, then by~\eqref{eq:5}, Lemma~\ref{l:2 variables functions} with~$\delta = \frac{\log 2}{3} \lambda(s^{\C})$, \eqref{eq:6}, and part~2 of Lemma~\ref{l:zero temperature tower},
\begin{displaymath}
  \begin{split}
\left| \hrho - \hrho' \right|
& \le
(n + 4) \Upsilon_2^t \left[ \tPi^+ (\tau, \lambda(s^{\C}))
- \sum_{\varsigma = s_0 - 3}^{s_0} \hJ_{\varsigma}^+(\tau, \lambda(s^{\C}))
- (|J_{s_0 - 1}| - s_0^2) \hJ_{s_0}^+(\tau, \lambda(s^{\C})) \right]
\\ & \le
(n + 4) \Upsilon_2^t 2^{-q s_0^2} \sum_{\varsigma = s_0 - 3}^{s_0} \hJ_{\varsigma}^-(\tau, \lambda(s^{\C})).        
  \end{split}
\end{displaymath}
Together with~\eqref{e:central tower section I} we get
$$ \left| \hrho - \hrho' \right|
\le
3(n + 4) \left(\Upsilon_1 \Upsilon_2 \right)^t 2^{- q s_0^2} |\hrho'|.
$$
Putting
$C=3(n+4)(\Upsilon_1\Upsilon_2)^{t_{*,\sup}}$
  we have
\begin{equation} 
\label{e:meas comparison}
\frac{\left| \hrho - \hrho' \right|}{|\hrho|}
\le \frac{\left| \hrho - \hrho' \right|}{|\hrho'|}
\le
C2^{- q s_0^2} \le C \exp(-qs_0^2 \log 2). 
\end{equation}

\partn{3} Using the inequalities~~$s_0\ge s^\C  \ge s^+(\tau)-2 $, $s^+(\tau_0)\ge 20$, $\Xi-2\xi\ge 1$, and the definitions of~$\tau$ and~$t_{*,\text{inf}}$ we have that for every $\varsigma\in [s_0-3,s_0]$,
\begin{equation}
\label{e:lower bound s_0}
 \varsigma^2
\ge
\frac{s_0^2}{2}
\ge
\frac{(s^+(\tau)-2)^2}{2} 
\ge
\frac{(s^+(\tau))^2}{4} 
= \frac{\Xi-2\xi}{4q(1-\tau)} 
\ge \frac{t_{*,\text{inf}}}{4q(t_*-t)}. 
\end{equation}
Using~\eqref{e:tau 0 2} and  putting~$\upsilon_! \= \upsilon_0 \frac{t_{*,\text{inf}}}{8q}$, we get
\begin{equation}
  \label{e:oscilating support}
C_0 \exp\left(\frac{ -\upsilon_0\cdot t_{*,\inf} }{4q(t_*-t)}\right) \le \exp\left(\frac{ - \upsilon_! }{t_*-t}\right).
\end{equation}

For $\varsigma\in \{+,-\},$ denote by~$\cO^\varsigma$ the forward orbit of~$p^\varsigma$  under~$f$.
Recall that  $\tau_{m_0-1} \le  \tau \le \tau_{m_0}$ and  $s^\C$ in $[s^+(\tau)-2,  s^-(\tau) +\Omega]$.
Using that $s^+$ and $s^-$ are increasing functions  we have 
$$
s^+(\tau_{m_0-1}) - 2 < s^+(\tau)-2 \le s^\C \le  s^-(\tau) +\Omega  \le s^-(\tau_{m_0})  +\Omega.
$$
Hence,
$$
s^+(\tau_{m_0-1}) - 2 < s_0   \le s^-(\tau_{m_0})  +\Omega+1.
$$
and by~\eqref{e:tau recursion} we
$$
s^+(\tau_{m_0-1}) - 2  < s_0 < s^+(\tau_{m_0+1}) - 5.
$$
This implies that  
\begin{equation}
\label{e:s_0 interval 1}
[s_0-3,s_0] \subseteq [\ell(m_0-1), \ell(m_0+1)).
\end{equation} 
A similar argument shows that if $\tau=\tau_{m_0-1}$ then 
\begin{equation}
\label{e:s_0 interval 2}
[s_0-3,s_0] \subseteq [\ell(m_0-1), \ell(m_0)).
\end{equation}  
Fix an integer~$\varsigma$ in~$[s_0 - 3, s_0]$ and 
let $m(\varsigma)$ in $\N_0$ be so that $\ell(m(\varsigma))\le \varsigma < \ell(m(\varsigma)+1)$.
For every integer~$j$ such that $j+1$ is in~$\bigcup_{k\in [\ell(m(\varsigma)),\ell(m(\varsigma)+1))} J_{k}$ we have~$\uwhx(\uvarsigma)_j = 1^{\varsigma(m(\varsigma))}$, so
$$ \iota(\uvarsigma)_j
=
\begin{cases}
1 & \text{if $\varsigma(m(\varsigma)) = +$};
\\
0 & \text{if $\varsigma(m(\varsigma)) = -$ and~$j$ is even};
\\
1 & \text{if $\varsigma(m(\varsigma)) = -$ and~$j$ is odd}.
\end{cases} $$
 Since $b_\varsigma$ is even, for every~$\ell$ in~$[0,a_{\varsigma + 1} - 1-b_\varsigma]$ the points
$ f^{n + 1 + 3(b_{\varsigma}+\ell-1)}(0)$ and $f^{3\ell}(p^{\varsigma(m(\varsigma))})$ are both in $Y_f$ or both in $\widetilde{Y}_f$.
Since $Y_f$ and $\widetilde{Y}_f$ are puzzle pieces of depth~4, it follows that 
$$ P_{f, 3(a_{\varsigma + 1} - 1-b_\varsigma)+ 4}(f^{n + 1 + 3(b_\varsigma-1)}(0))
=
P_{f, 3(a_{\varsigma + 1} - 1-b_\varsigma)+ 4}(p^{\varsigma(m(\varsigma))}). $$
Then, for each integer~$j$ in~$[b_{\varsigma}-1, a_{\varsigma+1} -2]$ we have
\begin{equation}\label{e:equal puzzle pieces}
 P_{f, 3(a_{\varsigma + 1}-2 -j)+4}(f^{n + 1 + 3j}(0))
=
P_{f, 3(a_{\varsigma + 1}-2 -j)+4}(f^{3j}(p^{\varsigma(m(\varsigma))})),
\end{equation}
which implies that  for each integer~$k$  in~$J_\varsigma$ and for each integer~$j$ in~$[b_{\varsigma}-1, k -1]$,
$$ P_{f, 3(k - j) +1}(f^{n + 1 + 3j}(0))
=
P_{f, 3(k - j) +1}(f^{3j}(p^{\varsigma(m(\varsigma))})).$$

Note that by definition of~$\fD_{ k}$, every element~$W$ of~$\fD_{ k}$ is contained in~$P_{f, n + 3k + 2}(0)$, so, if in addition we have $k\ge b_\varsigma+\varsigma^2$ and~$j \le  k  - \varsigma^2$, then by~\eqref{e:oscilating support} and Lemma~\ref{l:contractions} we obtain
$$ f^{n + 1 + 3j}(W) \cup  f^{(n + 1 + 3j) + 1}(W) \cup 
f^{(n + 1 + 3j) + 2}(W)
\subset
B(\cO^{\varsigma(m(\varsigma))}, \exp( - \upsilon_! / (t_*-t))),$$
which proves that~$\hrho'_\varsigma$ is supported on~$B(\cO^{\varsigma(m(\varsigma))}, \exp( - \upsilon_! / (t_*-t)))$.

On the other hand, for each integer~$k$ in~$J_{s_0}$,  every element $W$ of $\fD_{k}$ is contained in $P_{f,n+3k+2}(0)$, and hence in $P_{f,n+3(a_{s_0}-1)+2}(0)$. Thus, by~\eqref{e:equal puzzle pieces}
 with $\varsigma=s_0-1$  we have, by~\eqref{e:oscilating support} and Lemma~\ref{l:contractions}, that for every integer $j$ in $[b_{s_0-1}-1, a_{s_0}-s_0^2]$, 
  $$f^{n + 1 + 3j}(W) \cup  f^{(n + 1 + 3j) + 1}(W) \cup f^{(n + 1 + 3j) + 2}(W)
\subset
B(\cO^{\varsigma(m(s_0-1))} , \exp( - \upsilon_! / (t_*-t))), $$
which proves that~$\hrho''$ is supported on~$B(\cO^{\varsigma(m(s_0-1))}, \exp( - \upsilon_! / (t_*-t)))$.

Let~$m$ and~$\whm$ be integers as in the statement of Theorem~\ref{t:temperature dependence}, that is,
\begin{equation}
\label{e:emes}
 \whm \ge m \ge 1,
\varsigma(m) = \cdots = \varsigma(\whm), 
\text{ and fix } t \text{ in } [t_*\tau_m, t_* \tau_{\whm}].
\end{equation}
Then, $\tau_m \le \tau \le \tau_{\whm}$ so for every~$\varsigma$ in~$[s_0-3,s_0]$, by~\eqref{e:s_0 interval 1}
 if $\whm >  m$ and  by~\eqref{e:s_0 interval 2} if $\whm =  m$,  we have $\varsigma(m(\varsigma)) = \varsigma(m)$. 
It follows from the considerations above that the measure $\hrho'$ is supported on~$B(\cO^{\varsigma(m)}, \exp( - \upsilon_! / (t_*-t)))$. 
To conclude the proof, recall that the equilibrium state~$\rho_t$ of~$f|_{J(f)}$ for the potential $- t \log |D f|$ is the probability measure proportional to~$\hrho$.
By~\eqref{e:meas comparison} together with~\eqref{e:lower bound s_0} and~\eqref{e:oscilating support}, we conclude that 
$$ \rho_t\left(\C \setminus B\left(\cO^{\varsigma(m)},  \exp \left( -\frac{\upsilon_! }{t_*-t}  \right)\right)\right)
\le
\frac{|\hrho - \hrho'|}{|\hrho|}
\le
C \exp\left(- \frac{t_{*,\inf} \cdot \log 2}{2(t_*-t)}\right). $$
Recalling that $\upsilon_! = \upsilon_0 \cdot t_{*,\text{inf}}/8q$, this completes the proof of Theorem~\ref{t:temperature dependence} with 
$$\varkappa= \min\{\tfrac{\upsilon_0}{8q}, \tfrac{\log 2}{2}\} t_{*,\inf}.$$

\begin{rema}
\label{r:accumulation}
Without assuming the existence of $\whm$ and $m$ satisfying~\eqref{e:emes} the measure 
$\hrho'$ is supported on $B\left(\cO^{+}\cup \cO^{-},  \exp ( -\upsilon_!/(t_*-t) )\right)$, and the estimate above gives
that for every $t\in (	t_*\tau_1,t_*)$ we have
$$ \rho_t\left(\C \setminus B\left(\cO^{+}\cup \cO^{-},  \exp \left( -\frac{\varkappa}{t_*-t}  \right)\right)\right)
\le
C \exp\left(- \frac{\varkappa}{t_*-t}\right). $$

\end{rema}

\appendix

\section{Estimating the geometric pressure function}
\label{s:estimating pressure}

In this appendix we collect several results from~\cite{CorRiv1708} that we use in the proofs Theorems~\ref{t:pressure estimates} and~\ref{t:temperature dependence} in~\S\ref{s:proofs}.

Throughout this section we fix a uniform family of quadratic-like maps~$\sF$, with constants~$R$ and~$K$.

\subsection{Some uniform  bounds}
\label{ss:uniform distortion bound}

For a parameter~$c$ in~$\cP_2(-2)$ the external rays~$R_c(7/24)$ and~$R_c(17/24)$ land at the point~$\gamma(c)$ in~$P_{c,1}(0)$, see~\cite[\S3.3]{CorRiv13}.
Let~$\hU_c$ be the open disk containing~$-\beta(c)$ that is bounded by the
equipotential~$2$ and by
\begin{equation*}
\label{eq:pleasant cut}
R_c(7/24) \cup \{ \gamma(c) \} \cup R_c(17/24).  
\end{equation*}
Put $\hW_c \= f_c^{-1}(\hU_c)$, and for every $n\ge 3$ and every $f$ in $\cK_n(\sF)$ put $\hW_f \= h_f(\hW_{c(f)})$.

\begin{lemm}[Uniform distortion bound, \cite{CorRiv1708}, Lemma~4.9]
\label{l:distortion to central 0}
There is~$\Xi_2 > 1$ such that for each integer~$n \ge 4$ and each~$f$ in~$\cK_n(\sF)$ the following properties hold:
For each integer~$m \ge 1$ and each connected component~$W$ of~$f^{-m}(P_{f,1}(0))$ on which~$f^m$ is univalent, $f^m$ maps a neighborhood of~$W$ biholomorphically to~$\hW_f$ and the distortion of this map on~$W$ is bounded by~$\Xi_2$.
\end{lemm}

Recall that~$g_f \colon h_f(Y_{c(f)}\cup \tY_{c(f)}) \to P_{f, 1}(0)$ is defined by~$g_f = h_f \circ g_{c(f)}\circ h_f^{-1}$, see~\S\ref{ss:uniform families}.
\begin{lemm}[\cite{CorRiv1708}, Lemma~4.11]
  \label{l:contractions}
Let~$\sF$ be a uniform family of  quadratic-like maps.
Then, there are constants $C_0 > 0$ and $\upsilon_0 > 0$ such that for every~$f$ in~$\cP_5(\sF)$, every~$\ell$ in~$\N$, and every connected
component~$W$ of $g_f^{-\ell}(P_{f,1}(0))$, we have
$$
\max\{\diam(W), \diam(f(W)),\diam(f^2(W))\} \le C_0\exp(-\upsilon_0\ell).
$$
\end{lemm}

\subsection{Geometric Peierls condition}
\label{ss:criticallydriven}
Let~$n \ge 5$ be an integer and~$f$ in~$\cK_n(\sF)$.
Put 
$$
\pD' \= \{w\in \C \setminus
\pV \mid f^{m}(w)\in \pV \text{ for some } m \in \N\},
$$
and for~$w$ in $\pD'$ denote by~$m_f(w)$ the least integer~$m$ such that $f^{m}(w)\in \pV$ and
call it the \emph{first landing time of~$w$ to~$\pV$}.
The \emph{first landing map to~$\pV$} is the map~$\pL : \pD' \to \pV$ defined
by~$\pL(w) \= f^{m_f(w)}(w)$.

\begin{defi}[Geometric Peierls Condition]
\label{d:geometric Peierls}
Given~$\kappa > 0$ and $\upsilon > 0$, a quadratic-like map~$f$ in~$\sF$ satisfies the \emph{Geometric Peierls Condition with constants~$\kappa$ and $\upsilon$}, if for every~$z$ in~$L_f^{-1}(V_f)$ we have
\begin{equation*}
  \label{eq:7}
  |DL_f(z)| \ge \kappa \exp((\chicritf/2 + \upsilon) m_f(z)).
\end{equation*}
\end{defi}

\begin{prop}[\cite{CorRiv1708}, Proposition~4.3]
\label{p:transporting Peierls}
For every~$\upsilon > 0$ satisfying~$\upsilon < \frac{1}{2} \log 2$ and every~$R' >0$, there are constants~$K_1 > 1$, $n_1 \ge 6$, and~$\kappa_1 > 0$, such that the following property holds.
If~$\sF_1$ admits~$K_1$ and~$R'$ as uniformity constants, then for every integer~$n \ge n_1$, every element $f$ of~$\cK_n(\sF_1)$ satisfies the Geometric Peierls Condition with constants~$\kappa_1$ and~$\upsilon$, 
and we have
\begin{equation*}
  \label{eq:4}
  \chi_f(\beta(f)) > \chicritf + 2\upsilon,
  \end{equation*}
  \begin{equation}
  \label{e:choice of n for negative pressure}
  \chicritf > 2 \upsilon,
  \text{ and }
  \chi_f(p(f)) < \chi_f(p^+(f)) + \upsilon/4.
\end{equation}
\end{prop}

\subsection{Controlling the geometric pressure function via the postcritical series }

\begin{prop}[\cite{CorRiv1708}, Proposition~I]
\label{p:improved MS criterion}
Let $\sF$ be a uniform family of normalized quadratic-like maps. For every $\kappa>0$ and every $\upsilon>0$ 
there are~$n_2 \ge 5$ and~$C_1 > 1$
 such that for every 
integer~$n \ge n_2$ and every $f$ in~$\cK_n(\sF)$ satisfying the Geometric Peierls Condition with constants $\kappa$ and $\upsilon$,
the following properties hold for each~$t \ge  2\log2/\upsilon$.
\begin{enumerate}
\item[1.]
For~$p$ in~$[- t \pchicrit/2, 0)$ satisfying
$$ \sum_{k = 0}^{+ \infty} \exp(- (n + 3k)p)|Df^{n + 3k} (f(0))|^{-t/2}
\ge
C_1^{t}, $$
we have $\psP^{\R}(t, p) > 0$ and~$P_f^{\R}(t) \ge p$.
If in addition the sum above is finite, then~$\psP(t, p)$ is finite and~$P_f^{\R}(t) > p$.
\item[2.]
For~$p \ge - t \pchicrit/2$ satisfying
$$ \sum_{k = 0}^{+ \infty} \exp(- (n + 3k)p)|Df^{n + 3k} (f(0))|^{-t/2}
\le
C_1^{-t}, $$
we have $\psP(t, p) < 0$ and~$P_f(t) \le p$.
\item[3.]
For~$p \ge - t \pchicrit/2$ satisfying
$$ \sum_{k = 0}^{+ \infty} k \cdot \exp(- (n + 3k)p)|Df^{n + 3k} (f(0))|^{-t/2}
<
+ \infty, $$
we have
$$ \sum_{W \in \pfD} m_f(W) \cdot \exp(- m_f(W)p) \sup_{z \in W}|D\pF(z)|^{-t}
<
+ \infty. $$
\end{enumerate}
\end{prop}
Let~$n \ge 4$ be an integer and~$f$  in~$\cK_n(\sF)$.
Since the critical point~$z = 0$ does not belong to~$\pD$ (\emph{cf}., Lemma~4.2 in \cite{CorRiv13}), for each integer~$\ell \ge 1$, each connected component of~$\pD$ intersecting~$P_{f, \ell}(0)$ is contained
in~$P_{f, \ell}(0)$.
We define the \emph{level} of a connected component~$W$ of~$\pD$ as the largest integer~$k \ge 0$ such that~$W$ is contained in~$P_{f, n + 3k + 2}(0)$.
Given an integer~$k \ge 0$ denote by~$\fD_{f,k}$ the collection of all connected components of~$\pD$ of level~$k$; we have~$\pfD = \bigcup_{k = 0}^{+ \infty} \fD_{f,k}$.

\begin{prop}[\cite{CorRiv1708}, Proposition~5.6]
\label{p:estimating Z_1 by the postcritical series}
Let $\sF$ be a uniform family of normalized quadratic-like maps. For every $\kappa>0$ and every $\upsilon>0$ 
there are~$n_3 \ge 5$ and~$C_2 > 1$ such that for every integer~$n \ge n_3$
and 
every ~$f$ in~$\cK_n(\sF)$ satisfying the Geometric Peierls Condition with constants $\kappa$ and $\upsilon$, the following properties hold for each~$t \ge 2\log 2 / \upsilon$
and each integer $k \ge 0$:
\begin{enumerate}
\item[1.]
For each~$p < 0$, we have
\begin{multline*} 
\sum_{W \in \fD_{f,k}\cap \pfD^\R} \exp(-m_f(W)p) \inf_{z\in W} |D\pF(z)|^{-t}
>  \\
C_2^{-t} \exp(- (n + 3k)p)|Df^{n + 3k} (f(0))|^{-t/2}.
\end{multline*}
\item[2.]
For each~$p \ge - t \pchicrit/2 - t \upsilon/3$, we have
\begin{multline*} 
\sum_{W \in
\fD_{f,k}} \exp(-m_f(W)p) \sup_{z\in W} |D\pF(z)|^{-t}
<  \\
C_2^{t} \exp(- (n + 3k)p)|Df^{n + 3k} (f(0))|^{-t/2}.
\end{multline*}
Moreover,
\begin{multline*} 
\sum_{W \in
\fD_{f,k}} (m_f(W)-(n+3k)) \exp(-m_f(W)p) \sup_{z\in W} |D\pF(z)|^{-t}
<  \\
C_2^{t} \exp(- (n + 3k)p)|Df^{n + 3k} (f(0))|^{-t/2}.
\end{multline*}

\end{enumerate}
\end{prop}

\subsection{Estimating the postcritical series}
\label{ss:estimating postcritical series}
Denote by~$\hSigma$ the set of all those sequences~$(\whx_j)_{j \in \N_0}$ in~$\{0, 1^+, 1^- \}^{\N_0}$ such that for each~$j$ in~$\N_0$ satisfying~$\whx_j = 1^+$ (resp.~$\whx_j = 1^-$), we have~$\whx_{j + 1} \neq 1^-$ (resp.~$\whx_{j + 1} \neq 1^+$).
A sequence~$(x_j)_{j \in \N_0}$ in~$\{0, 1 \}^{\N_0}$ is \emph{compatible} with a sequence~$(\whx_j)_{j \in \N_0}$ in~$\hSigma$ if for every~$j$ in~$\N_0$ such that~$\whx_j = 0$ (resp. $\whx_j = 1^+$, $\whx_j = \whx_{j + 1} = 1^-$), we have~$x_j = 0$ (resp. $x_j = 1$, $x_j \neq x_{j + 1}$).

Fix a sequence~$(\whx_j)_{j \in \N_0}$ in~$\hSigma$.
Define~$N : \N_0 \to \N_0$ by~$N(0) \= 0$, and for~$k$ in~$\N$ by
$$ N(k)
\=
\sharp \{ j \in \{0, \ldots, k-1 \} \mid \whx_j = 0 \}. $$
Moreover, define~$B : \N_0 \to \N_0$ by~$B(0) \= 0$, $B(1) \= 1$, and for~$k\ge 2$ by
$$ B(k)
\=
1 + \sharp \{ j \in \{0, \ldots, k - 2 \}
 \mid \whx_j \neq \whx_{j + 1} \}. $$
Note that for~$k$ in~$\N$ the function~$B(k)$ is equal to the number of blocks of~$0$'s, $1^+$'s, and~$1^-$'s in the sequence~$(\whx_j)_{j = 0}^{k-1}$.

Throughout the rest of this subsection, fix a uniform family of quadratic-like maps~$\sF$, and let~$\Xi_2 > 1$ be the constant given by Lemma~\ref{l:distortion to central 0}.
Note that for~$n \ge 5$ and~$f$ in~$\cK_n(\sF)$, the condition~$\theta(f) > 1$ is equivalent to~$\chi_f(p(f)) > \chi_f(\wtp(f))$, see~\eqref{d:theta n} for the definition of~$\theta(f)$.
When this holds, define
$$ \xi(f)
\=
\frac{\log \Xi_2}{2 \log \theta(f)}. $$
If in addition the itinerary~$\iota(f)$ is compatible with~$(\whx_j)_{j \in \N_0}$, let~$\xi \ge \xi(f)$ be given and define for each integer~$k \ge 0$ and each~$(\tau, \lambda)$ in~$[0, + \infty) \times [0, + \infty)$,
\[
 \pi_{f,k}^{\pm}(\tau,\lambda)
\=
2^{- \lambda k - \tau N(k) \pm \xi \tau B(k)}.
\]

\begin{lemm}[\cite{CorRiv1708}, Lemma~6.1]
\label{l:2 variables functions}
There  is~$\Xi_3>1$ such that the following property holds. Let~$\uwhx$ be a sequence in~$\hSigma$ such that~$N(k)/k \to 0$ as~$k \to + \infty$, and such that the length of every maximal block  containing only $1^-$'s is even.
Then, for every integer~$n \ge 6$ and every~$f$ in~$\cK_n(\sF)$ such that
$ \chi(\whp(f)) = \chi(\wtp(f)), $
and such that the itinerary~$\iota(f)$ is compatible with~$\uwhx$, 
we have
 \begin{equation*}\label{e:chicrit}
\chicritf
=
\frac{1}{3} \log |Dg_f(\wtp(f))|.
 \end{equation*}
If moreover 
$$ \chi_f(p(f)) > \chi_f(\wtp(f))$$
the following property holds for every choice of~$\xi \ge \xi(f)$.
For every integer~$k \ge 0$, and all $t>0$ and $\delta \ge 0$, we have
\begin{multline*}
\label{eq:2 variables series}
\Xi_3^{- \frac{t}{2}} \exp(-n\delta)\left( \frac{\exp(\chicritf)}{|Df(\beta(f))|}
\right)^{\frac{t}{2}n }
\pi_{f, k}^- \left( \frac{\log \theta(f)}{\log 2} t, \frac{3 \delta}{\log 2}
\right)
\\ 
\begin{aligned}
& \le
\exp \left(-(n+3k) \left(- t \frac{\chicritf}{2} + \delta \right) \right)
|Df^{n+3k}(f(0))|^{-\frac{t}{2}}
\\ & \le
\Xi_3^{\frac{t}{2}} \exp(-n\delta) \left( \frac{\exp (\chicritf)}{|Df(\beta(f))|}
\right)^{\frac{t}{2} n }
\pi_{f, k}^+ \left( \frac{\log \theta(f)}{\log 2} t, \frac{3 \delta}{\log 2}
\right).
\end{aligned}
\end{multline*}
\end{lemm}

\section{Estimating the positive temperature~2 variables series}
\label{s:reduction}
\subsection{Preliminary estimates}
\label{ss:positive temperature partition}
Given an integer~$\Xi \ge 0$, let~$q \ge 100 (\Xi + 1)$ be a sufficiently large integer such that~$2^{q - 3} \ge q + 1 + \Xi$.
Define the cubic function
$$ \begin{array}{cccl}
 Q : & \R & \to & \R \\
       &  s  & \mapsto & Q(s) \= q s^3.
\end{array} $$
For each real number~$s$ in~$[0,+\infty)$ define 
\[
a_s \= 2^{Q(s)} \text{ and } b_s\= 2^{Q(s)} + Q(s + 1) - Q(s)  + \Xi,
\]
and the following intervals of~$\R$:
\[
I_s
\=
\left[ a_s, b_s \right)
\text{ and }
J_s
\=
\left[ b_s, a_{s+1} \right).
\]
Note that~$|I_0| = q + \Xi$, and that for integer values of~$s$, the intervals~$I_s$ and~$J_s$ form a partition of $[1,+\infty)$ that we use in~\S\ref{ss:subfamily} to  define a certain family of itineraries.
For~$s$ in~$[0,+\infty)$ that is not necessarily an integer, the interval~$J_s$ is used in the proof of Lemmas~\ref{l:positive temperature first floor}   and~\ref{l:zero temperature tower} in~\S\ref{ss:estimating the 2 variables series} and~\S\ref{ss:estimating the weighted 2 variables series}, respectively.

Define the function $N : \N_0 \to \N_0$, by $N(0) \= 0$, and for~$k$ in~$\N$ by
\[
N(k)
\=
\sharp \left\{ j \in \{0, \ldots, k-1 \} \mid j + 1 \in \bigcup_{s=0}^{+\infty} I_s \right\},
\]
and the function~$B : \N_0 \to \N_0$ by $B(0) \= 0$, and for~$s$ in~$\N_0$ by
\begin{equation}
  \label{e:B}
B^{-1}(2s + 1) = I_s
\quad \text{and} \quad
B^{-1}(2(s + 1)) = J_s.
\end{equation}

Observe that for every~$s$ in~$\N_0$, we have for every~$k$ in~$J_s$
\begin{equation}
\label{eq:N J}
N(k)
=
\sum_{j = 0}^s |I_j|
=
\sum_{j = 0}^s (Q(j + 1) - Q(j) + \Xi)
 =
Q(s+1) + \Xi \cdot (s+1)
\end{equation}
and for every~$k$ in~$I_s$
\begin{equation}
\label{eq:N I}
N(k) = k + 1 - a_s + Q(s) + \Xi s.
\end{equation}

Let~$\xi > 0$ be given, put~$\Xi \= \lceil 2 \xi \rceil + 1$, and let~$q$ be as  before.
For~$s$ in~$\N_0$ define the following~2 variables series on~$[0,+\infty) \times [0,+\infty)$,
\begin{equation*}
I_s^{\pm}(\tau, \lambda)
\=
\sum_{k\in  I_s} 2^{- \lambda k -\tau N(k) \pm \tau \xi B(k)}
\quad \text{and} \quad
J_s^{\pm}(\tau, \lambda)
\=
\sum_{k\in J_s} 2^{- \lambda k -\tau N(k) \pm \tau \xi B(k)},
\end{equation*}
and put
$$ \Pi^{\pm}(\tau,\lambda)
\=
1 + \sum_{s=0}^{+\infty} I_s^{\pm}(\tau, \lambda) + \sum_{s=0}^{+\infty} J_s^{\pm}(\tau, \lambda). $$
Note that by~\eqref{e:B} and~\eqref{eq:N J}, for all~$j$ in~$\N_0$ and all $\tau > 0$ we have
\begin{equation}
\label{e:J formula}
J_j^{\pm}(\tau,\lambda)
=
2^{-  \tau Q(j+1) - (\Xi \mp 2 \xi) \tau \cdot (j+1)} \sum_{k \in J_j} 2^{- \lambda k}.
\end{equation}
For every real number~$s$ in~$[0,+\infty)$ define
\[
\lambda(s) \= \frac{1}{|J_s|}.
\]
By part~2 of Lemma~\ref{l:the itinerary} (below) and the hypothesis~$q \ge 100(\Xi+1)$, we have 
\begin{equation}
\label{e:lambda upper bound}
0 < \lambda(s) \le 1/4.
\end{equation}

In~\S\ref{ss:estimating the 2 variables series} we estimate $\Pi^+$ from above and $\Pi^-$ from below for some particular values of the variable $\lambda$ depending on the variable $\tau$.
In~\S\ref{ss:estimating the weighted 2 variables series} we estimate a weighted version of the series $\Pi^+$ from above.

We finish this subsection with a lemma collecting some simple estimates that are extensively used in what follows.
\begin{lemm}
\label{l:the itinerary}
The  following properties hold:
\begin{enumerate}
 \item[1.] For each real number~$s \ge 0$, we have
$b_s \le a_{s+1}/2$ and the function $\lambda(s)$ is strictly decreasing;
 \item[2.] For each real number~$s \ge 0$, we have
$a_{s+1}/2 \le |J_s|$;
\item[3.] For every $s\ge 1$, we have $b_s + 2(s+1)^2\le \frac{5}{4} a_s$;
\item[4.] Given $s\ge 1$, put $s_0 \= \lceil s \rceil$. Then, we have $b_{s_0} + 2s_0^2\le 3 |J_s|$.
\end{enumerate}
\end{lemm}
\begin{proof}
Part~1, case~$s = 0$ is given by our hypothesis~$2^{q - 3} \ge q + 1 + \Xi$.
For~$s > 0$, it follows from this and from the fact that, by the hypothesis $q\ge 100(\Xi+1)$, the derivative of the function
$$ s \mapsto 2^{q(s + 1)^3 - 1} - (2^{qs^3} + q(3s^2 + 3s + 1) + \Xi) $$
is strictly positive on~$[0, + \infty)$. This also implies that $\lambda(s)$ is strictly decreasing.

Part~2 follows from part~1.

Part~3, case~$s = 1$ is given by our hypotheses~$q\ge 100(\Xi+1)$ and~$2^{q - 3} \ge q + 1 + \Xi$. For~$s > 1$, it follows from this and from the fact that the derivative of the function
$$ s \mapsto \frac{5}{4} 2^{qs^3} - (2^{qs^3} + q(3s^2 + 3s + 1) + \Xi+2(s+1)^2) $$
is strictly positive on~$[1, + \infty)$.

Part~4 follows from parts~2 and~3.
\end{proof}

\subsection{Estimating the~2 variables series}
\label{ss:estimating the 2 variables series}

Define the functions~$s^+,s^- : (- \infty, 1) \to \R$ by
$$ s^\pm(\tau) \= \left(\frac{\Xi \mp 2\xi}{q(1 - \tau)}\right)^{\frac{1}{2}}. $$

\begin{lemm}
\label{l:positive temperature first floor}
For every~$\tau$ in~$\left( \tfrac{q-1}{q}, 1 \right)$, we have
\[
 \Pi^+ \left( \tau , \lambda \left( s^+(\tau) - 2 \right) \right)
\le 25+
5 \cdot 2^{\tau \xi},
\]
and for every~$\Omega \ge 0$, we have
\[
2^{2\Omega - 3}
\le
\Pi^- \left(\tau , \lambda \left(s^-(\tau) + \Omega \right) \right).
\]

\end{lemm}
\begin{sublemma}\label{subl:11}
The following properties hold:
\begin{enumerate}
\item[1.]
For $\tau \ge 1$ we have
\[
\Pi^+(\tau,0) \le 2 \left( 2^{\tau\xi} + 1 \right).
\]
\item[2.]
For every real number~$s$ in~$[0,+\infty)$ and every~$\tau$ in~$(2/3,1)$
satisfying~$\tau > \frac{Q(s+1)-1}{Q(s+2)}$, we have
\begin{equation*}
\Pi^+(\tau,\lambda(s))
\\ \le
4 + 5 \cdot 2^{\tau\xi}
+ 2 \sum_{j=0}^{\lfloor s \rfloor + 1}  2^{(1-\tau)Q(j+1) -\tau (\Xi - 2\xi)
(j+1) }.
\end{equation*}
\item[3.]
For every real number~$s$ in~$[0,+\infty)$ and every~$\tau > 0$, we have
\[
\frac{1}{8} 2^{(1-\tau)Q(\lfloor s \rfloor + 1)-\tau(\Xi + 2\xi)(\lfloor s
\rfloor + 1)}
\le
\Pi^-(\tau,\lambda(s)).
\]
\end{enumerate}
\end{sublemma}
\proof
\partn{1}
By~\eqref{e:B}, \eqref{eq:N I}, and the hypothesis~$\Xi - 2 \xi \ge 1$, for every~$\tau > 0$ and every~$\lambda \ge 0$ we have
\begin{equation}
\label{eq:Is+1}
\begin{split}
\sum_{s=0}^{+\infty} I_s^+(\tau, \lambda)
& \le
\sum_{s=0}^{+\infty}  \sum_{m = 1}^{|I_s|} 2^{-\tau(Q(s) + \Xi s + m) +
\tau \xi \cdot (2s + 1)}
\\ & =
2^{\tau\xi} \sum_{s=0}^{+\infty} 2^{-\tau(Q(s) + (\Xi - 2 \xi)s) }
\sum_{m = 1}^{|I_{s}|} 2^{-\tau m}
\\ & \le
2^{\tau\xi} \frac{2^{-\tau}}{1-2^{-\tau}} \sum_{s=0}^{+\infty} 2^{-\tau (\Xi - 2\xi)s}
\\ & \le
2^{\tau\xi} \frac{ 2^{-\tau}}{\left(1-2^{-\tau}\right)^2}.
  \end{split}
\end{equation}
By~\eqref{e:B}, \eqref{eq:N J}, the hypothesis~$\Xi - 2\xi \ge 1$, and that for every $s \ge 0$ we have $|J_s|\le a_{s + 1}$, we obtain for every~$\tau \ge 1$ and with $\lambda=0$
\begin{equation}
\label{eq:J}
\begin{split}
\sum_{s=0}^{+\infty} J_s^+(\tau, 0)
= &
\sum_{s=0}^{+\infty} |J_s|2^{-\tau (Q(s+1) + \Xi\cdot(s+1)) + 2 \tau \xi
\cdot (s+1)}
\\ \le &
\sum_{s=0}^{+\infty} 2^{-(\tau-1)Q(s+1) - \tau (\Xi - 2 \xi) (s+1) }
\\ \le &
\frac{2^{-\tau}}{1-2^{-\tau}}.
  \end{split}
\end{equation}
Combining inequalities~\eqref{eq:Is+1} and~\eqref{eq:J}, we get for every~$\tau \ge 1$
\begin{equation*}
\Pi^+(\tau, 0)
\le
1 + 2^{\tau\xi} \frac{2^{-\tau}}{\left( 1-2^{-\tau} \right)^2}
+ \frac{2^{-\tau}}{1-2^{-\tau}}
\le
2 \left( 2^{\tau\xi} + 1 \right).
\end{equation*}
This is part~1 of the sublemma.

\partn{2}
Fix~$s$ in $[0,+\infty)$ and set $s_0 \= {\lfloor s \rfloor}$.
We use~\eqref{eq:Is+1} to estimate~$\Pi^+(\tau, \lambda(s))$.
To estimate~$\sum_{j=0}^{+\infty}J_j^+(\tau, \lambda(s))$, note that by definition
of~$\lambda(s)$, for each integer~$\ell$ satisfying~$1 \le \ell \le |J_s|$ we
have
\[
\frac{1}{2} \le  2^{- \lambda(s) \ell} \le 1.
\]
On the other hand, the inequality $q \ge 100(\Xi+1)$ implies that the function~$j \mapsto
|J_j|$ is nondecreasing on~$[0, + \infty)$.
Therefore, for each~$j$ in~$\{0, \ldots, s_0 \}$ we have~$|J_j| \le |J_s|$ and
then
\begin{equation}
\label{eq:geometrica1}
\frac{1}{2}|J_j|
\le
\sum_{m = 1}^{|J_j|} 2^{- \lambda(s) m}
\le
|J_j|.
\end{equation}
On the other hand
\begin{equation}
\label{eq:geometrica2}
\sum_{m = 1}^{+ \infty} 2^{- \lambda(s) m}
=
\frac{1}{2^{\lambda(s)} - 1}
\le
\frac{1}{\lambda(s) \log 2}
\le
2 |J_s|.
\end{equation}
Note also that for $j\in \N$, by~\eqref{e:J formula}, the inequalities $b_j\ge b_0 = q+\Xi+1 \ge 1$, \eqref{eq:geometrica2} and $|J_s| \le 2^{Q(s+1)}$, we have
\begin{equation}
\label{eq:10}
\begin{split}
J_j^+(\tau,\lambda(s))
& =
2^{-\tau(Q(j+1) + (\Xi - 2 \xi)(j + 1))} \sum_{k \in J_j} 2^{- \lambda(s)k}
  \\ & \le 
2 |J_s| 2^{-\tau(Q(j+1) + (\Xi - 2 \xi)(j + 1))}
\\ & \le 
2 \cdot 2^{Q(s + 1)-\tau(Q(j+1) + (\Xi - 2 \xi)(j + 1))}.
  \end{split}
\end{equation}
Taking~$j = s_0 + 1$ and using the inequality~$Q(s + 1) \le Q(s_0 + 2)$, we
obtain,
\begin{equation}
\label{eq:intermediate term}
J_{s_0 + 1}^+(\tau, \lambda(s))
\le
2 \cdot 2^{(1 - \tau) Q(s_0 + 2) - \tau (\Xi - 2 \xi)(s_0 + 2)} .
\end{equation}
On the other hand, our hypothesis $\tau > \frac{Q(s + 1) - 1}{Q(s + 2)}$ implies that for $j \ge s_0+2$ we have
\[
Q(s + 1) - \tau Q(j + 1) \le Q(s + 1) - \tau Q(s + 2) <  1.
\]
So, using the inequalities~$\Xi - 2 \xi \ge 1$ and $\tau > 2/3$, and summing~\eqref{eq:10} over~$j$
satisfying~$j \ge s_0 + 2$, we obtain
\begin{equation}
\label{eq2}
\sum_{j = s_0 + 2}^{+\infty} J_j^+(\tau,\lambda(s))
\le
\sum_{j=s_0 + 2}^{+\infty} 2^{2 -\tau (\Xi - 2\xi) (j+1)}
\le
\frac{2^{2-3\tau  }}{1-2^{-\tau }}
< 
\frac{1}{1-2^{-\tau }}.
\end{equation}
Now we complete the estimate of~$\sum_{j=0}^{+\infty}J_j^+(\tau, \lambda(s))$, by
estimating the terms for which~$j$ is in~$\{0, \ldots, s_0 \}$.
From~\eqref{eq:geometrica1}, \eqref{e:J formula}, and $|J_j| \le
2^{Q(j+1)}$, we deduce that for every integer~$j$ in~$\{0, \ldots, s_0 \}$ we
have
\begin{equation*}
J_j^+(\tau, \lambda(s))
 \le
|J_j| \cdot 2^{-\tau (Q(j+1) + (\Xi - 2 \xi)(j + 1))}
 \le
2^{(1-\tau )Q(j+1) -\tau (\Xi - 2 \xi)(j+1)}.
\end{equation*}
Summing over~$j$ in~$\{0, \ldots, s_0 \}$ and using
inequalities~\eqref{eq:intermediate term} and~\eqref{eq2}, we obtain
\begin{multline*}
 \sum_{j=0}^{+\infty}J_j^+(\tau, \lambda(s))
\\ 
\begin{aligned}
& \le
\sum_{j=0}^{s_0} 2^{(1-\tau )Q(j+1) -\tau (\Xi - 2 \xi)(j+1)}
+ 2 \cdot 2^{(1 - \tau) Q(s_0 + 2) - \tau (\Xi - 2 \xi)(s_0 + 2)}
+ \frac{ 1}{1-2^{-\tau }}
\\ & \le
2 \sum_{j=0}^{s_0 + 1} 2^{(1-\tau )Q(j+1) -\tau (\Xi - 2\xi)(j+1)}
+ \frac{ 1}{1-2^{-\tau }}.
\end{aligned}
\end{multline*}
Together with~\eqref{eq:Is+1} this implies
\begin{multline}
\label{e:preliminary upper bound}
\Pi^+(\tau ,\lambda(s))
\\ \le
1 + \frac{ 2^{-\tau }}{\left( 1-2^{-\tau } \right)^2} 2^{\tau \xi}
+
2 \sum_{j=0}^{s_0 + 1} 2^{(1-\tau )Q(j+1) -\tau (\Xi - 2\xi)(j+1)}
+ \frac{ 1}{1- 2^{-\tau } }.
\end{multline}
Using  that~$\tau$ is in~$(2/3, 1)$, we have~$\frac{ 1}{ 1-2^{-\tau } } \le 3$ and $\frac{ 2^{-\tau }}{\left( 1-2^{-\tau } \right)^2} \le 5$.
We obtain part~2 of the sublemma by combining these estimates
with~\eqref{e:preliminary upper bound}.

\partn{3}
Fix~$s$ in~$[0, + \infty)$ and set~$s_0 \= {\lfloor s \rfloor}$.
By parts~1 and 2 of Lemma~\ref{l:the itinerary} we have
$$
\lambda(s)(b_{s_0}-1) \le  2 \cdot \frac{b_{s_0}}{a_{s+1}} \le 1.
$$
By this inequality,  part~2 of Lemma~\ref{l:the itinerary},  \eqref{e:J formula}, the first inequality of~\eqref{eq:geometrica1}, and the first equality in~\eqref{eq:10}, we have
\begin{multline*}
 \frac{1}{2^3} 2^{(1-\tau )Q(s_0+1)-\tau (\Xi + 2 \xi)(s_0+1)}
\\
\begin{aligned}
 & \le
\frac{1}{2^2} |J_{s_0}| 2^{- \tau (Q(s_0+1) + \Xi \cdot (s_0+1)) - 2 \tau  \xi \cdot (s_0 + 1)}
\\ & \le
\frac{1}{2} |J_{s_0}| 2^{- \lambda(s) (b_{s_0} - 1) -\tau(Q(s_0+1) + \Xi \cdot (s_0+1)) - 2 \tau \xi \cdot (s_0 + 1)}
\\ & \le
\left(\sum_{m = 1}^{|J_{s_0}|} 2^{- \lambda(s) m} \right) 2^{- \lambda(s) (b_{s_0} - 1) -\tau(Q(s_0+1) + \Xi \cdot (s_0+1)) - 2 \tau \xi \cdot (s_0 + 1)}
\\ & =
 J_{s_0}^-(\tau, \lambda(s)).
\end{aligned}
\end{multline*}
This proves part~3 of the sublemma and completes the proof.
\endproof

\proof[Proof of Lemma~\ref{l:positive temperature first floor}]
Fix~$\tau$ in~$\left( \tfrac{q-1}{q}, 1 \right)$ and $\Omega\ge 0$. Put $s(\tau)=\left(\tfrac{2}{q(1-\tau)}\right)^{\frac{1}{2}} $.
Note that~$\tau > \tfrac{q-1}{q} \Leftrightarrow 2< s(\tau)^2$. This last inequality and our hypothesis $q\ge 100(\Xi+1)$ imply that  
$$(3q+2) s(\tau)  <   3qs(\tau)^2 + q +1. $$
Finally the last inequality is equivalent to
\begin{eqnarray*} 
2s(\tau)  & <  & \,Q(s(\tau)) - Q(s(\tau)-1) +1 \quad   \Leftrightarrow \\
(1-\tau)Q(s(\tau)) &   <  & Q(s(\tau)) - Q(s(\tau)-1) +1 \quad   \Leftrightarrow \\
\frac{Q\left(s(\tau)- 1\right) - 1}{Q\left(s(\tau)\right)} &<  &\tau. 
\end{eqnarray*}
On the other hand, the function~$s \mapsto \frac{Q(s - 1) - 1}{Q(s)}$ is strictly increasing on $[0,+\infty)$.
Using $1\le \Xi -2\xi\le 2$ we deduce 
$$
\frac{Q(s^+(\tau) - 1) - 1}{Q(s^+(\tau))}
\le
\frac{Q\left(s(\tau) - 1\right) - 1}{Q(s(\tau))}
<
\tau.$$
So the hypotheses of part~2 of Sublemma~\ref{subl:11} are satisfied with~$s =
s^+(\tau) - 2$.
Let~$F : \R \to \R$ be the quadratic function defined by
$$ F(\ell) \= (1 - \tau) Q(\ell) - \tau (\Xi - 2\xi) \ell. $$
Note that $F(0) = 0$,
$$ F \left( \frac{s^+(\tau)}{2} \right)
=
\frac{(\Xi - 2\xi)^\frac{3}{2}}{8q^{\frac{1}{2}}}(1-\tau)^{\frac{1}{2}} - \frac{3\tau(\Xi - 2\xi)}{4} \left( \frac{s^+(\tau)}{2}
\right)
$$ 
and 
$$
F(s^+(\tau)) = \frac{(\Xi - 2\xi)^\frac{3}{2}}{q^{\frac{1}{2}}}(1-\tau)^{\frac{1}{2}}. $$
Using that~$F$ is convex on $[0,+\infty)$, we conclude that for each~$\ell$ in~$[0, s^+(\tau)]$ we
have
\begin{multline*}
F(\ell) = (1 - \tau) Q(\ell) - \tau (\Xi - 2 \xi) \ell \\
\le
\frac{(\Xi - 2\xi)^\frac{3}{2}}{q^{\frac{1}{2}}}(1-\tau)^{\frac{1}{2}} - \frac{3\tau(\Xi - 2\xi)}{4} \min \left\{ \ell, s^+(\tau) - \ell \right\}.
\end{multline*}
Therefore, putting~$s^+ = s^+(\tau) - 2$ and using~$1 \le \Xi - 2\xi \le 2$,~$q \ge 100(\Xi+1)$, and~$\tau\ge \frac{q-1}{q} $, we have
\begin{displaymath}\label{e: first sum}
  \begin{split}
\sum_{j=0}^{\lfloor s^+ \rfloor + 1} 2^{(1- \tau)Q(j+1) - \tau (\Xi - 2\xi) (j+1)}
& \le
2 \sum_{\ell = 0}^{\lfloor s^+ \rfloor + 2 } 2^{\frac{(\Xi - 2\xi)^\frac{3}{2}}{8 \sqrt{2}}(1-\tau)^{\frac{1}{2}} - \frac{3\tau(\Xi - 2\xi)}{4} \ell}
\\ & \le
2 \cdot 2^{\frac{1}{40}} \frac{1}{1-2^{-\frac{3}{4}\cdot \frac{2}{3}}}
\\ & <
10.    
  \end{split}
\end{displaymath}
The first inequality of the lemma is then obtained using part~2 of
Sublemma~\ref{subl:11} with~$s = s^+$.

To prove the second inequality,
first we define~$F^- : \R \to \R$ as the quadratic function given by
$$ F^-(\ell) \= (1 - \tau) Q(\ell) - \tau (\Xi + 2\xi) \ell. $$
Notice from the definition of $\Xi=\lceil 2\xi\rceil +1$ and the hypothesis that $\tau$ is in $\left(\tfrac{q-1}{q},1\right)$, we have
\begin{displaymath}
  \begin{split}
F^-(s^-(\tau) + \Omega)
  = \;&(1-\tau)\cdot(\Xi+2\xi)s^-(\tau) +
(3 - \tau)\cdot (\Xi + 2\xi)  \Omega \;  
\\ \; & + q\left(1 - \tau\right)  \Omega^2 \left(3s^-(\tau) + \Omega\right)
\\  \ge \; &
2(\Xi + 2\xi) \Omega
\\  \ge \;&
2\Omega.
  \end{split}
\end{displaymath}
Using that~$F^-$ is increasing on the interval~$\left[\left(\frac{\tau}{3}\right)^{\frac{1}{2}}  s^-(\tau), + \infty \right)$, that contains~$s^-(\tau)$, 
and putting~$s^- = s^-(\tau) + \Omega$, we get
\begin{equation*}
2\Omega
\le
F^-(s^-)
\le
F^-({\lfloor s^- \rfloor} + 1)
=
(1 - \tau) Q({\lfloor s^- \rfloor} + 1) - \tau (\Xi +2\xi) ({\lfloor s^-
\rfloor} + 1).
\end{equation*}
Together with part~3 of Sublemma~\ref{subl:11} with~$s = s^-$, we obtain
\[
2^{2\Omega}
\le
2^{(1 -\tau )Q(\lfloor s^- \rfloor + 1) - \tau (\Xi + 2\xi)(\lfloor s^- \rfloor
+ 1)}
\le
8 \Pi^-\left(\tau , \lambda \left(s^- \right) \right),
\]
proving the second inequality of the lemma.
\endproof

\subsection{Estimating the weighted~2 variables series}
\label{ss:estimating the weighted 2 variables series}

For each~$s$ in~$\N_0$, $\tau > 0$, and~$\lambda \ge 0$ put
\begin{equation*}
  \begin{split}
\tI_s^+(\tau, \lambda)
& \=
\sum_{k\in  I_s} k \cdot 2^{- \lambda k -\tau N(k) + \tau \xi B(k)},
\\
\tJ_s^+(\tau, \lambda)
& \=
\sum_{k\in J_s} k \cdot 2^{- \lambda k -\tau N(k) + \tau \xi B(k)},    
  \end{split}
\end{equation*}
and
$$ \tPi^+(\tau,\lambda)
\=
1 + \sum_{s=0}^{+\infty} \tI_s^+ (\tau, \lambda)
+ \sum_{s=0}^{+\infty} \tJ_s^+ (\tau, \lambda). $$

Noting that by part~2 of Lemma~\ref{l:the itinerary} we have~$a_{s + 1} - b_s = |J_s| \ge s^2 + 1$, define for each~$\tau > 0$ and~$\lambda\ge 0$,
$$ \hJ_s^\pm(\tau, \lambda)
\=
\sum_{k = b_s + s^2}^{a_{s + 1} - 1} (k + 1 - b_s - s^2) \cdot 2^{- \lambda k -\tau N(k) \pm \tau \xi B(k)}. $$
\begin{lemm}
\label{l:zero temperature tower}
For each~$\tau$ in $[1/2,1]$, the following properties hold:
\begin{enumerate}
\item[1.] For each $s>0$ we have
$ \tPi^+(\tau, \lambda(s)) < + \infty $;
\item[2.]
Given~$s\ge 10$, put~$s_0 \= \lceil s \rceil$.
Then
\begin{multline*}
\Pi^+(\tau, \lambda(s))  \le \tPi^+(\tau, \lambda(s)) 
- \sum_{\varsigma = s_0 - 3}^{s_0} \hJ_{\varsigma}^+(\tau, \lambda(s)) 
- (|J_{s_0 - 1}| - s_0^2) J_{s_0}^+(\tau, \lambda(s)) \\
\le
2^{- qs_0^2} \sum_{\varsigma = s_0 - 3}^{s_0} \hJ_{\varsigma}^-(\tau, \lambda(s)).
\end{multline*}
\end{enumerate}
\end{lemm}
The proof of this lemma is after the following lemma.
\begin{sublemma}
\label{l:positive temperature core}
Given~$s\ge 10$, put $s_0\=\lceil s \rceil$. For every  $\tau$ in $[1/2,1]$   the following properties hold:
\begin{enumerate}
\item[1.]
$ \hJ_{s_0}^-(\tau, \lambda(s))
\ge
\frac{1}{2^{9}} 2^{2q(s + 1)^3 - q \tau \cdot (s_0+1)^3 - (\Xi +2 \xi) \tau \cdot (s_0+1)}; $
\item[2.] 
For every integer~$\varsigma$ in~$[s_0 - 3, s_0 -1]$ we have
$$(b_\varsigma + \varsigma^2)  J_{\varsigma}^+(\tau, \lambda(s))
\le
 \frac{1}{20}2^{-qs_0^2} \hJ_{\varsigma}^-(\tau, \lambda(s)); $$
\item[3.]
$ (b_{s_0} - |J_{s_0 - 1}| + 2s_0^2) J_{s_0}^+(\tau, \lambda(s))
\le
\frac{1}{4}2^{-qs_0^2} \hJ_{s_0}^-(\tau, \lambda(s)). $
\end{enumerate}
\end{sublemma}
\begin{proof}

\partn{1}
By~\eqref{e:B}, \eqref{eq:N J}, and part~4 of Lemma~\ref{l:the itinerary}, we have
\begin{multline}
\label{eq:tgeometrica}
\hJ^-_{s_0}(\tau,\lambda(s))
\\
\begin{aligned}
& =
2^{- q \tau \cdot (s_0+1)^3 - (\Xi + 2 \xi) \tau \cdot (s_0+1)}
\sum_{k = b_{s_0}+s_0^2 }^{a_{s_0 + 1} - 1} (k+1-b_{s_0}-s_0^2) \cdot 2^{- \lambda(s)k}
\\ & \ge
\frac{1}{2^3} 2^{- q \tau \cdot (s_0+1)^3 - (\Xi + 2 \xi) \tau \cdot (s_0+1)}
\sum_{m = 1}^{|J_{s_0}|-s_0^2}m \cdot 2^{- \lambda(s) m}.
  \end{aligned}
\end{multline}
Noticing that for every integer~$N \ge 1$ we have
$$ \sum_{m = 1}^{N} m \cdot 2^{- \lambda(s) m}
=
\frac{2^{\lambda(s)}}{(2^{\lambda(s)} - 1)^2}
\left( 1 - (N + 1) 2^{- \lambda(s) N} + N 2^{- \lambda(s) (N + 1)} \right), $$
and that the function
$$ \eta \mapsto 1 - (N + 1) \eta^N + N \eta^{N + 1} $$
is decreasing on~$[0, 1]$, we have 
\begin{multline*}
\sum_{m = 1}^{|J_{s_0}|-s_0^2} m \cdot 2^{- \lambda(s) m}
\\
\begin{aligned}
& =
\frac{2^{\lambda(s)}}{(2^{\lambda(s)} - 1)^2}
\cdot \left( 1 - (|J_{s_0}| -s_0^2 + 1) 2^{- \lambda(s_0)( |J_{s_0}|-s_0^2)} +   (|J_{s_0}| -s_0^2)2^{- \lambda(s_0) (|J_{s_0}| -s_0^2 + 1)} \right)
\\ & =
\frac{2^{\lambda(s)}}{(2^{\lambda(s)} - 1)^2}
\cdot \left( 1 - 2^{ \lambda(s_0)s_0^2-1} -  ( |J_{s_0}|-s_0^2)2^{ \lambda(s_0)s_0^2-1}  \left( 1 - 2^{- \lambda(s_0)} \right) \right)
\\ & =
\frac{2^{\lambda(s)}}{(2^{\lambda(s)} - 1)^2}
\cdot \left( 1 - 2^{ \lambda(s_0)s_0^2-1}\left( 1 +  \left( |J_{s_0}|-s_0^2\right)  \left( 1 - 2^{- \lambda(s_0)} \right)\right) \right).
\end{aligned}
\end{multline*}
Observe that  $1-2^{-\lambda(s_0)}\le \log (2^{\lambda(s_0)})$.  Using this inequality  we obtain
\begin{multline}
\label{e:core main estimate}
\sum_{m = 1}^{|J_{s_0}|-s_0^2} m \cdot 2^{- \lambda(s) m}
\ge
  \frac{2^{\lambda(s)}}{(2^{\lambda(s)} - 1)^2} \left( 1 - 2^{\lambda(s_0)s_0^2-1}(1 + \log 2) \right) 
\\ \ge
\frac{1}{2^4} \frac{1}{(2^{\lambda(s)} - 1)^2}.
\end{multline}
Note that by~$0\le \lambda(s) \le 1$, we have~$2^{\lambda(s)} - 1 \le  \lambda(s)$.
Thus, together with part~2 of Lemma~\ref{l:the itinerary}, the previous chain of inequalities implies
$$ \sum_{m = 1}^{|J_{s_0}| - s_0^2} m \cdot 2^{- \lambda(s) m}
\ge
\frac{1}{2^4} \cdot |J_s|^2
\ge
\frac{1}{2^6} \cdot 2^{2q (s + 1)^3}. $$
Together with~\eqref{eq:tgeometrica}, the inequality $\Xi \ge 2 \xi$, and our hypothesis~$q \ge 100 (\Xi + 1)$, this implies
\begin{displaymath}
\hJ_{s_0}^-(\tau,\lambda(s))
\ge
\frac{1}{2^{9}} 2^{2q(s + 1)^3 - q \tau \cdot (s_0+1)^3 - (\Xi + 2 \xi) \tau \cdot (s_0+1)}.
\end{displaymath}
This proves part~1.

\partn{2}
Let~$\varsigma$ be an integer in~$[s_0 - 3, s_0]$ and note that by~\eqref{e:B}, \eqref{e:J formula}, and the definition of~$\hJ^-_{\varsigma}$, we have
\begin{equation}
\label{e:core quotient formula}
\frac{J_{\varsigma}^+(\tau, \lambda(s))}{\hJ_{\varsigma}^-(\tau, \lambda(s))}
=
2^{4 \tau \xi \cdot (\varsigma + 1) + \lambda(s) \varsigma^2} \frac{\sum_{m = 1}^{|J_{\varsigma}|} 2^{-\lambda(s) m}}{\sum_{m = 1}^{|J_{\varsigma}| - \varsigma^2} m \cdot 2^{-\lambda(s) m}}.
\end{equation}
Suppose~$\varsigma$ is in~$[s_0 - 3, s_0 - 1]$.
Then~$\lambda(s) |J_{\varsigma}| \le 1$, so 
\begin{displaymath}
  \begin{split}
\frac{J_{\varsigma}^+(\tau, \lambda(s))}{\hJ_{\varsigma}^-(\tau, \lambda(s))}
& \le
2 \cdot 2^{4 \tau \xi \cdot (\varsigma + 1) + \lambda(s) \varsigma^2} \frac{\sum_{m = 1}^{|J_{\varsigma}|} 2^{- \lambda(s) m} }{\sum_{m = 1}^{|J_{\varsigma}| - \varsigma^2} m}
\\ & \le
2^2 \cdot 2^{4 \tau \xi \cdot (\varsigma + 1) + \lambda(s) \varsigma^2} \frac{|J_{\varsigma}|}{(|J_{\varsigma}| - \varsigma^2)^2}.    
  \end{split}
\end{displaymath}
Noting that by part~2 of Lemma~\ref{l:the itinerary} and our hypotheses $s\ge 10$ and~$q \ge 100 (\Xi + 1)$    we have
$$ \lambda(s) \varsigma^2 \le \varsigma^2 / |J_\varsigma| \le 1
\quad \text{and} \quad
|J_{\varsigma}| \le 2(|J_{\varsigma}| - \varsigma^2).$$
By parts~2 and~3  of Lemma~\ref{l:the itinerary}, by the inequality~$\Xi \ge 2 \xi$, by our hypotheses~$q \ge 100 (\Xi + 1)$, $\tau \in [1/2,1]$ and~$s \ge 10$, we have
\begin{displaymath}
  \begin{split}
(b_\varsigma+\varsigma^2) \frac{J^+_{\varsigma}(\tau, \lambda(s))}{\hJ^-_{\varsigma}(\tau, \lambda(s))}
& \le
2^6 a_\varsigma 2^{4\tau \xi \cdot(\varsigma+1)} |J_{\varsigma}|^{-1}
\\ & \le  2^{q\varsigma^3 + q (\varsigma+1) - q (\varsigma + 1)^3}
\\ & \le
2^{-q(3\varsigma^2+2\varsigma)}
\\ & \le
2^{- 3 q \varsigma^2}
\\ & \le
\frac{1}{20}2^{- q s_0^2}.    
  \end{split}
\end{displaymath}
This proves part~2.

\partn{3} 
By part~3 of Lemma~\ref{l:the itinerary} and the hypothesis $s\ge 10$  we have
\begin{displaymath}
b_{s_0} - |J_{s_0 - 1}| + 2 s_0^2
=
b_{s_0}-(a_{s_0}-b_{s_0-1}) + 2s_0^2
\le
b_{s_0-1} + 2s_0^2
\le
2\cdot 2^{q(s_0 - 1)^3}.
\end{displaymath}
Thus, by part~2 of Lemma~\ref{l:the itinerary}, \eqref{e:core main estimate}, \eqref{e:core quotient formula}, and the inequality~$\lambda(s) \le 1$, we have
\begin{displaymath}
  \begin{split}
(b_{s_0} - |J_{s_0 - 1}| + 2 s_0^2) \frac{J_{s_0}^+(\tau, \lambda(s))}{\hJ_{s_0}^-(\tau, \lambda(s))}
& \le
2^6 \cdot 2^{q(s_0 - 1)^3 + 4 \tau \xi \cdot (s_0 + 1) + \lambda(s) s_0^2} (2^{\lambda(s)} - 1)
\\ & \le
2^6 \lambda(s) \cdot 2^{q(s_0 - 1)^3 + 4 \tau \xi \cdot (s_0 + 1) + \lambda(s) s_0^2}
\\ & \le
2^7 \cdot 2^{- q (s + 1)^3 + q(s_0 - 1)^3 + 4 \tau \xi \cdot (s_0 + 1) + \lambda(s) s_0^2}.    
  \end{split}
\end{displaymath}
Using~$\lambda(s) s_0^2 = s_0^2 |J_s|^{-1} \le 1$, the inequality~$\Xi \ge 2\xi$, and our hypotheses~$q \ge 100 (\Xi + 1)$, $\tau\in [1/2,1]$ and~$s \ge 10$, we have
\begin{displaymath}
(b_{s_0} - |J_{s_0 - 1}| + 2 s_0^2) \frac{J_{s_0}^+(\tau, \lambda(s))}{\hJ_{s_0}^-(\tau, \lambda(s))}
\le
2^{- q s_0^3 + q(s_0 - 1)^3 + q (s_0 + 1)}
\le
\frac{1}{4}2^{- q s_0^2}.
\end{displaymath}
This completes the proof of part~3 and of the lemma.
\end{proof}

\proof[Proof of Lemma~\ref{l:zero temperature tower}]
\partn{1}
Note that for every~$s \ge 0$, we have~$\lambda(s) \le 1$ and
\begin{displaymath}
\sum_{m = 1}^{+ \infty} m \cdot  2^{- \lambda(s) m}
=
\frac{2^{\lambda(s)}}{\left(2^{\lambda(s)} - 1 \right)^2}
\le
\frac{2^{\lambda(s)}}{(\lambda(s) \log 2)^2}
\le
2^3 |J_s|^2
\le
\left( 2^3 \right) 2^{2q(s+1)^3}.
\end{displaymath}
Together with~\eqref{e:B}, \eqref{eq:N J}, \eqref{eq:N I}, for every~$j$ in~$\N_0$ this implies
\begin{multline}
\label{eq:t10 2}
\tJ_j^+(\tau,\lambda(s))
+ \tI_{j + 1}^+(\tau,\lambda(s))
\\
\begin{aligned}
& \le
2^{-\tau(q(j+1)^3 + \Xi \cdot (j+1)) +  \tau\xi \cdot (2j+3)}
\sum_{k \in J_j \cup I_{j + 1}} k \cdot 2^{- \lambda(s)k}
\\ & \le 
(2^{\tau\cdot\xi + 3}) 2^{2q(s + 1)^3 - q \tau \cdot (j+1)^3 - \tau (\Xi-2\xi) \cdot (j + 1)}.
\end{aligned}
\end{multline}
By the inequality~$\Xi - 2\xi \ge 1$ and  our hypothesis that~$\tau$ is in  $[1/2,1]$,  for every~$j \ge 2s+1$ we have
$$ \tJ_j^+(\tau,\lambda(s))
+ \tI_{j + 1}^+(\tau,\lambda(s))
\le
\left( 2^{ \xi + 3} \right) 2^{- \tau \cdot (j + 1)}. $$
This implies that~$\tPi^+(\tau, \lambda(s))$ is finite, as wanted.

\partn{2}
The first inequality follows directly from the definitions.
Note that by~\eqref{e:B} and~\eqref{eq:N I}, and our hypothesis that~$\tau$ is in  $[1/2,1]$, we have
\begin{equation}
\label{eq:first t block 2}
1+ \tI_0^+(\tau, \lambda(s))
\le
1 + \sum_{k = 1}^{+ \infty} k \cdot 2^{- \tau k +  \xi}
=
1+ 2^{\xi} \frac{2^{- \tau}}{(1 - 2^{- \tau})^2}
\le
10\cdot 2^{\xi}.
\end{equation}
By part~3 of Lemma~\ref{l:the itinerary}, \eqref{e:B}, \eqref{eq:N I}, the inequality~$\Xi - 2\xi \ge 1$ and our hypothesis that~$\tau$ is in  $[1/2,1]$,  for every integer $j\ge 1$ we have
\begin{displaymath}
  \begin{split}
\tI_{j}(\tau, \lambda(s))
& \le
\left( 2^{ \xi} \right) 2^{- \tau (q j^3 + (\Xi - 2 \xi)j)}
\sum_{m = 1}^{|I_{j}|} \left( a_j + m \right) 2^{- \tau m}
\\ 
& \le
\left( 2^{ \xi} \cdot \frac{5}{4} \right) 2^{- \tau (q j^3 + (\Xi - 2 \xi)j)}
a_j \sum_{m = 1}^{|I_{j}|}  2^{- \tau m}
\\
& \le
\left( 2^{ \xi + 2} \right)
2^{q j^3 - q \tau j^3 - (\Xi - 2 \xi) \tau j} 
\\ & \le
\left( 2^{\xi + 2} \right)
2^{q (1 - \tau) j^3 - \tau j}.    
  \end{split}
\end{displaymath}
Combined with~\eqref{eq:first t block 2} and our hypothesis that~$\tau$ is in  $[1/2,1]$, this implies
\begin{equation*}
  \label{eq:total tI}
1+ \sum_{j = 0}^{s_0 + 1}  \tI_{j}^+(\tau, \lambda(s))
\le
\left( 2^{ \xi + 5 }\right) 2^{(1-\tau)q (s_0 + 1)^3}.
\end{equation*}
Together with part~1 of Lemma~\ref{l:positive temperature core},  the inequality~$\Xi \ge 2 \xi$ and our hypotheses~$s \ge 10$ and $q\ge 100(\Xi+1)$, this  inequality implies
\begin{equation}
    \label{eq:first t I 2}
1+ \sum_{j = 0}^{s_0 + 1}  \tI_{j}^+(\tau, \lambda(s))
\le
\frac{1}{2^3} 2^{- q s_0^2} \hJ_{s_0}^-(\tau, \lambda(s)).
\end{equation}

On the other hand, by~\eqref{e:B} and~\eqref{eq:N J} for every~$j$ in~$\{0, \ldots, s_0 - 4 \}$ we have
\begin{displaymath}
  \begin{split}
\tJ_j^+(\tau,\lambda(s))
& =
2^{-\tau(q(j+1)^3 + \Xi \cdot (j+1)) + 2 \tau \xi \cdot (j+1)}
\sum_{k \in J_j} k \cdot 2^{- \lambda(s)k}
\\ & \le 
|J_j|2^{q(j + 1)^3 - q \tau \cdot (j+1)^3 - (\Xi - 2 \xi) \tau \cdot (j + 1)}
\\ & \le 
2^{2q(j + 1)^3 - q \tau \cdot (j+1)^3 - (\Xi - 2 \xi) \tau \cdot (j + 1)}
\\ & \le
2^{2q(s_0 -3)^3 - q \tau \cdot (s_0-3)^3 - (\Xi - 2 \xi) \tau \cdot (s_0 -3)}.
  \end{split}
\end{displaymath}
Together with part~1 of Lemma~\ref{l:positive temperature core}, the inequality $\Xi\ge 2\xi$ and with our hypotheses~$q \ge 100 (\Xi + 1)$ and ~$s\ge 10$, this implies
\begin{equation}
  \label{eq:first t Js 2}
\frac{\sum_{j = 0}^{s_0 - 4} \tJ_j^+(\tau, \lambda(s))}{\hJ_{s_0}^-(\tau, \lambda(s))}
\le
s_0 2^{- 2 q s_0^2}
\le
\frac{1}{2^2}  2^{-   q s_0^2}.
\end{equation}

On the other hand, by~\eqref{eq:t10 2}, part~1 of Lemma~\ref{l:positive temperature core}, and our hypotheses~$q \ge 100 (\Xi + 1)$ and $\tau \in [1/2,1]$, for every integer~$j \ge s_0 + 1$ we have
\begin{multline*}
\frac{\tJ_j^+(\tau, \lambda(s)) + \tI_{j + 1}^+(\tau, \lambda(s))}{\hJ_{s_0}^-(\tau, \lambda(s))}
\\
  \begin{aligned}
& \le
\left( 2^{\xi +12} \right) 2^{-q \tau \cdot ((j + 1)^3 - (s_0+1)^3) - (\Xi-2\xi)\tau\cdot(j + 1)+ (\Xi+2\xi)\tau\cdot(s_0 + 1)}
\\ & \le
\left( 2^{ \xi +12} \right) 2^{-q \tau \cdot ((j - s_0)^3+3(j-s_0)^2(s_0+1) + 3(j-s_0)(s_0+1)^2) + 4\xi\tau(s_0 + 1)}
\\ & \le
\frac{1}{2^4} 2^{-q s_0^2(j - s_0)}.
  \end{aligned}
\end{multline*}
Summing over~$j \ge s_0 + 1$ and using our hypotheses~$q \ge 100 (\Xi + 1)$ and~$s \ge 10$, we obtain
\begin{equation*}
\frac{\sum_{j = s_0 + 1}^{+ \infty} \left( \tJ_j^+(\tau, \lambda(s)) + \tI_{j + 1}^+(\tau, \lambda(s)) \right)}{\hJ_{s_0}^-(\tau, \lambda(s))}
\le
\frac{1}{2^4} \frac{2^{-q s_0^2}}{1 - 2^{-q s_0^2}}
\le
 \frac{1}{2^3}2^{- q s_0^2}.
\end{equation*}
Combined with~\eqref{eq:first t I 2} and~\eqref{eq:first t Js 2}, this implies 
\begin{equation}\label{e:total measure minus core towers}
\tPi^+(\tau, \lambda(s)) - \sum_{\varsigma = s_0 - 3}^{s_0} \tJ_{\varsigma}^+(\tau, \lambda(s)) \le
\frac{1}{2} 2^{- qs_0^2} \hJ_{s_0}^-(\tau, \lambda(s)).
\end{equation}

Notice that for each positive integer $\varsigma$ we have 
\begin{multline}\label{e:difference tower total measure}
  \tJ_{\varsigma}^+(\tau, \lambda(s)) -  \hJ_{\varsigma}^+(\tau, \lambda(s))
  \\ 
\begin{aligned}
& \le   \sum_{k=b_\varsigma}^{b_\varsigma+\varsigma^2-1} k \cdot 2^{- \lambda k -\tau N(k) + \tau \xi B(k)} +
\sum_{k = b_\varsigma + \varsigma^2}^{a_{\varsigma + 1} - 1} ( b_\varsigma + \varsigma^2) \cdot 2^{- \lambda k -\tau N(k)  + \tau \xi B(k)} \\ 
& \le  (b_\varsigma + \varsigma^2)  J_\varsigma^+(\tau,\lambda(s)).
\end{aligned}
\end{multline}
Using this and part~2 of Lemma~\ref{l:positive temperature core} for each integer $\varsigma$ in $[s_0-3, s_0-1]$, we get
\begin{equation}\label{e:difference tower total measure in s}
\tJ_{\varsigma}^+(\tau, \lambda(s)) -  \hJ_{\varsigma}^+(\tau, \lambda(s)) 
 \le  \frac{1}{20}2^{-qs_0^2} \hJ_{\varsigma}^-(\tau, \lambda(s)).
\end{equation}
Now by~\eqref{e:difference tower total measure}  with $s=s_0$ and part~3 of Lemma~\ref{l:positive temperature core}  we have 
\begin{multline}\label{e:difference tower total measure in s_0}
\tJ_{s_0}^+(\tau, \lambda(s))- \hJ_{s_0}^+(\tau, \lambda(s))- (|J_{s_0-1}|-s_0^2)J^+_{s_0}(\tau,\lambda(s)) \le \\   \frac{1}{4} 2^{-qs_0^2} \hJ_{s_0}^-(\tau, \lambda(s)).
\end{multline}
Finally, putting  together~\eqref{e:total measure minus core towers},  \eqref{e:difference tower total measure in s}  and~\eqref{e:difference tower total measure in s_0} we obtain the desired inequality and complete the proof of the lemma.
\endproof

\bibliographystyle{alpha}

\end{document}